\documentclass[a4paper,11pt]{amsart}
\usepackage{amsthm,amssymb,mathtools,amsfonts,latexsym,rawfonts,amsmath, mathrsfs, lscape}
\usepackage{stmaryrd,tikz,pgfplots}
\usepackage{marginnote}  
\usepackage[backref=page]{hyperref}
\usepackage[top=1.5in, bottom=0.9in, left=0.9in, right=0.9in]{geometry}
\usepackage{hyperref}
\usepackage[textsize=tiny]{todonotes}
\usepackage{bigints}
\usepackage{bm}

\newcommand{\Marginpar}[1]{\marginpar{\tiny{#1}}}
\newcommand{\Note}[1]{{\par\noindent\hrulefill\par\tiny{#1}\par\noindent\hrulefill\par}}
\newcommand{\Detail}[1]{{#1}}
\renewcommand{\Marginpar}[1]{}
\renewcommand{\Note}[1]{}
\renewcommand{\Detail}[1]{}
\renewcommand*{\backref}[1]{}
\renewcommand*{\backrefalt}[4]{%
    \ifcase #1 (Not cited.)%
    \or        (Cited on page~#2.)%
    \else      (Cited on pages~#2.)%
    \fi}

\usepackage{hyperref}
\usepackage[all]{xy}
\usepackage{tensor}

\usepackage{array, tabularx}

\usepackage{booktabs} 
\usepackage{siunitx}

\usepackage{tikz-cd}

\allowdisplaybreaks

\newtheorem{thm}{Theorem}
\newtheorem*{thm*}{Theorem}
\newtheorem{prop}[thm]{Proposition}
\newtheorem{lem}[thm]{Lemma}
\newtheorem*{lem*}{Lemma}
\newtheorem{cor}[thm]{Corollary}
\newtheorem*{cor*}{Corollary}

\theoremstyle{definition}
\newtheorem{defn}[thm]{Definition}
\newtheorem*{defn*}{Definition}
\newtheorem{ex}[thm]{Example}
\newtheorem{rem}[thm]{Remark}

\renewcommand{\[}{\begin{equation*}}
\renewcommand{\]}{\end{equation*}}

\DeclareMathOperator\tr{tr}

\def \tr {\mbox{tr}}

\renewcommand{\bar}{\overline}

\begin{document}
\parskip1mm

\title[Generalized almost-K\"ahler--Ricci solitons]{Generalized almost-K\"ahler--Ricci solitons}

\author{Michael Albanese}
\address{Pure Mathematics, University of Waterloo, Waterloo, ON N2L 3G1, Canada.}
\email{m3albane@uwaterloo.ca}

\author{Giuseppe Barbaro}
\address{Institut for Matematik, Aarhus University, Ny Munkegade 118, 8000 Aarhus C, Denmark.}
\email{g.barbaro@math.au.dk}

\author{Mehdi Lejmi}
\address{Department of Mathematics, Bronx Community College of CUNY, Bronx, NY 10453, USA.}
\email{mehdi.lejmi@bcc.cuny.edu}

\thanks{The second author is supported by GNSAGA of INdAM. The third author is supported by the Simons Foundation Grant \#636075.}



\maketitle

\begin{abstract}
	We generalize K\"ahler--Ricci solitons to the almost-K\"ahler setting as the zeros of Inoue's moment map~\cite{MR4017922}, and show that their existence is an obstruction to the existence of first-Chern--Einstein almost-K\"ahler metrics on compact symplectic Fano manifolds.
	We prove deformation results of such metrics in the $4$-dimensional case. Moreover, we study the Lie algebra of holomorphic vector fields on $2n$-dimensional compact symplectic Fano manifolds admitting generalized almost-K\"ahler--Ricci solitons. In particular, we partially extend Matsushima's theorem~\cite{MR0094478} to compact first-Chern--Einstein almost-K\"ahler manifolds.
\end{abstract}
  
\section{introduction}

	On a symplectic manifold $(M,\omega)$ of real dimension $2n$, an almost-complex structure $J$ is compatible with $\omega$ if $\omega(J\cdot, J\cdot) = \omega(\cdot, \cdot)$ and the $2$-tensor $g:=\omega(\cdot,J\cdot)$ induced by $\omega$ and $J$ is a Riemannian metric -- we call such a metric an almost-K\"ahler metric. 
	If $J$ is integrable, then $g$ is a K\"ahler metric; if $J$ is not integrable, we say $g$ is strictly almost-K\"ahler. On an almost-K\"ahler manifold, the canonical Hermitian connection $\nabla$~\cite{MR66733,MR0165458,MR1456265} is the unique connection preserving the almost-K\"ahler structure and whose torsion is $J$-anti-invariant (equivalently, has no $(1,1)$-part). The curvature of the induced canonical Hermitian connection on the anti-canonical bundle is of the form $\sqrt{-1}\rho^\nabla$, where $\rho^\nabla$ is a real $2$-form which represents $2\pi c_1(M).$ The trace of $\rho^\nabla$ with respect to $\omega$ is the Chern scalar curvature $s^c$. When $g$ is K\"ahler, $\nabla$ is the Levi-Civita connection, $\rho^{\nabla}$ is the Ricci form, and $s^c$ coincides with the (half of) Riemannian scalar curvature.

	We consider the space $AK_\omega$ of almost-complex structures compatible with $\omega.$ It turns out that
	the natural action of $Ham(M, \omega)$, the Hamiltonian symplectomorphism group, on $AK_\omega$ is Hamiltonian with moment map given by $s^c$~\cite{MR1073369,MR1622931}. It follows that the zeros of the moment map correspond to constant Chern scalar curvature almost-K\"ahler metrics. In particular, in the K\"ahler setting, they are constant scalar curvature K\"ahler metrics.
	These results have been extended to a more general setting in~\cite{MR4017922}.
	In that work, Inoue fixed a compact subgroup $G$ of $Ham(M,\omega)$, and considered the group $Ham^G(M,\omega)$ of Hamiltonian symplectomorphisms commuting with $G$, and the space $AK^G_\omega$ of $G$-invariant almost-complex structures compatible with $\omega$.
	He proved that the action of $Ham^G(M,\omega)$ on $AK^G_\omega$ is Hamiltonian with moment map given by
	$$s^c_\xi:=s^c-n-2\Delta^gf_\xi+2f_\xi-2|df_\xi|_g^2,$$ 
	where $\Delta^g$ is the Riemannian Laplacian with respect to the metric $g$ and $f_\xi$ is a (normalized) Hamiltonian potential of $\xi$ which is a fixed element in the center of the Lie algebra of $G$. 
	On a compact K\"ahler Fano manifold, Inoue proved that the zeros of this moment map correspond to K\"ahler--Ricci solitons~\cite[Proposition 3.2]{MR4017922}. 
	Consequently, motivated by the analogy with K\"ahler--Ricci solitons, if an almost-K\"ahler metric is induced by an almost-complex structure in $AK^G_\omega$, we call it a {\it{generalized almost-K\"ahler--Ricci soliton}} (GeAKRS for short) if it satisfies the condition $$s^c_\xi\equiv 0.$$ 
		
	If the vector field $\xi$ is trivial, a GeAKRS reduces to an almost-K\"ahler metric with constant Chern scalar curvature.
	On the other hand, we prove (in Proposition~\ref{positivity}) that the existence of GeAKRS with non-trivial $\xi$ on a compact symplectic monotone manifold (i.e. $2\pi c_1(M)=\lambda[\omega],$ for some constant $\lambda$) implies that $\lambda$ is positive, and then $(M,\omega)$ is a symplectic Fano manifold. 
	Working on symplectic Fano manifolds, we provide a generalization of~\cite[Proposition 3.2]{MR4017922} to the almost-K\"ahler setting. 
	Indeed, in Proposition~\ref{GeAKRS-exact-dc}, we show that on a compact symplectic Fano manifold with $\rho^\nabla-\omega=da,$ for some $G$-invariant $1$-form $a$, the existence of a GeAKRS is equivalent to the fact that the $d^c:=JdJ^{-1}$-exact part of the Hodge decomposition of $a$ with respect to a twisted Laplacian is given by $-d^cf_\xi$. 
	Furthermore, it turns out that on a compact symplectic Fano manifold, the existence of a GeAKRS with non-trivial $\xi$ is an obstruction to the existence of a constant Chern scalar curvature almost-K\"ahler metric in $AK^G_\omega$.
	Consequently, we obtain an obstruction to the existence of first-Chern--Einstein almost-K\"ahler metrics (i.e. almost-K\"ahler metrics satisfying $\rho^\nabla=\lambda\omega,$ for constant $\lambda$), just as K\"ahler--Ricci solitons are obstructions to the existence of K\"ahler--Einstein metrics (see for example~\cite{MR1768112,MR1817785,MR2483362}).
	\begin{thm*}[Theorem~\ref{obstruction-constant-chern}]
		Let $(M,\omega)$ be a compact symplectic Fano manifold. Suppose that there exists $J\in AK^G_\omega$ which induces a GeAKRS with respect to a vector field $\xi$ and $\tilde{J}\in AK^G_\omega$ which induces an almost-K\"ahler metric of constant Chern scalar curvature. Then $\xi\equiv 0.$
	\end{thm*}

	In order to construct non-trivial examples, we study the linearized operator associated to the GeAKRS equation, and, in the spirit of \cite{MR1053910,MR1274118} (see also \cite{MR2661166,MR2807093,MR3663320,MR3897025}), we prove that these metrics can be obtained as deformations of K\"ahler--Ricci solitons.
	More precisely, we show a deformation result in real dimension $4$ along any smooth path of almost-complex structures $J_t$ compatible with $\omega$ which satisfies a technical hypothesis on  $h^-_{J_t}$, the dimension of harmonic $J_t$-anti-invariant $2$-forms.
	\begin{thm*}[Theorem~\ref{deformations}]
		Let $(M,\omega,J,g)$ be a $4$-dimensional compact K\"ahler manifold such that the metric is a K\"ahler--Ricci soliton
		with respect to the vector field $\xi$. Let $T$ be a maximal torus in $Ham(M,\omega)$ such that the Lie algebra of $T$
		contains $\xi$. Let $J_t$ be any smooth family of $T$-invariant, $\omega$-compatible almost-complex structures such that $J_0=J$ and for small $t$, we have $h^-_{J_t}=b^+-1.$ Then, there exists a family of $T$-invariant, $\omega$-compatible almost-complex structures $\tilde{J}_t$ such that $(\omega,\tilde{J}_t)$ is GeAKRS for small enough $t$ and
		$\tilde{J}_0=J_0$. Moreover, $\tilde{J}_t$ is diffeomorphic to $J_t$ for each $t.$
	\end{thm*}
	We also observe that in the toric case the deformation argument can be made without any assumption on $h_{J_t}^-$ (see Lemma \ref{toric-deformations}). Consequently, since there exist K\"ahler--Ricci solitons on any toric K\"ahler Fano manifold~\cite{MR2084775} we obtain that
	\begin{cor*}[Corollary~\ref{existence-toric-dim4}]
		Any $2n$-dimensional compact toric symplectic Fano manifold admits a GeAKRS which is strictly almost-K\"ahler.
	\end{cor*}
	By continuing the analogy with the K\"ahler case, we discuss the Lie algebra of (real) holomorphic vector fields on compact $2n$-dimensional almost-K\"ahler manifolds admitting GeAKRS. 
	As a matter of fact, in the K\"ahler setting, a structure theorem for the automorphism group on manifolds admitting K\"ahler--Ricci solitons is given in~\cite{MR1768112} (see also \cite{MR0094478,MR0124009,MR0645743,MR3897025} for different canonical metrics). 
	In particular, Killing vector fields are Hamiltonian on such manifolds. 
	We extend some of these results to symplectic Fano GeAKRS manifolds.
	\begin{thm*}[Theorem~\ref{killing-algebra}]
		Let $(M,\omega,J,g)$ be a compact symplectic Fano GeAKRS manifold of real dimension $2n > 2$ such that $\rho^\nabla-\omega=-dd^cf_\xi.$ Let $X$ be a (real) holomorphic vector field. Then, the Riemannian dual of the vector field $X$ is of the form$$\alpha=d^cu+dh,$$ where $u,h$ are functions normalized by $\int_Mue^{-2f_\xi}\omega^n=\int_Mhe^{-2f_\xi}\omega^n=0$ and satisfying
		\begin{eqnarray}
			\frac{1}{2}\Delta^gu& =&u-g\left(\alpha,d^cf_\xi\right),\label{result1}\\
			\frac{1}{2}\Delta^g h&=&h-g(J\alpha,d^cf_\xi)\label{result2},\\
			D^g_{(dh)^\sharp}\omega&=&0.
		\end{eqnarray}
		where $\Delta^{{g}}$ is the Riemannian Laplacian of the metric $g$ and $\sharp$ is the $g$-Riemannian dual. Furthermore, if $X$ is a Killing vector field, then $X$ is a Hamiltonian vector field such that the Killing potential $u$ satisfies 
		\begin{equation*}
			\Delta^{\tilde{g}}u=2u,
		\end{equation*}
		where $\Delta^{\tilde{g}}$ is the Riemannian Laplacian of the conformal metric $\tilde{g}=e^{\frac{-2f_\xi}{n-1}}g.$ In particular, there are no non-trivial Killing vector fields if $2$ is not an eigenvalue of $\Delta^{\tilde{g}}.$
	\end{thm*}
	We remark that in the K\"ahler case if we have a K\"ahler-Ricci soliton and functions $u,h$ satisfying Equations~(\ref{result1}) and~(\ref{result2}), then the vector field $X$ obtained as the Riemannian dual of $\alpha=d^cu+dh$ is a holomorphic vector field (see~\cite{MR1768112}).
	Therefore, from the above result, it is possible to deduce a structure theorem of the Lie algebra of holomorphic vector fields as in~\cite{MR1768112}. 
	However, in the strictly almost-K\"ahler case, more assumptions are needed in order for $X$ to be a holomorphic vector field (see Remark~\ref{remark1} for details). 
	The proof of Theorem~\ref{killing-algebra} suggests that we can partially extend M Matsushima's theorem~\cite{MR0094478} for K\"ahler--Einstein manifolds to compact first-Chern--Einstein almost-K\"ahler manifolds. 
	\begin{cor*}[Corollary~\ref{AK-Matsushima}]
		Let $(M,\omega,J,g)$ be a compact almost-K\"ahler manifold such that the metric $g$ is a first-Chern--Einstein metric. Then, we have the following
		\begin{enumerate}
			\item If $\rho^\nabla=0$, then any holomorphic vector field is a Killing vector field. Moreover,
			the Lie algebra of holomorphic vector fields is abelian and the Riemannian dual of any holomorphic vector field is harmonic with respect to the twisted Laplacian
			$\Delta^c=J\Delta^gJ^{-1}.$\\
			\item If $\rho^\nabla=-\omega$, then there are no non-trivial holomorphic vector fields on $M.$\\
			\item If $\rho^\nabla=\omega$, then the Riemannian dual $\alpha$ of a holomorphic vector field $X$ is of the form
			$$\alpha=d^cu+dh,$$
			such that 
			\begin{eqnarray*}
				\Delta^gu=2u,\\ 
				\Delta^gh=2h,\\
				D^g_{(dh)^\sharp}\omega=0.
			\end{eqnarray*}
			In particular, there are no non-trivial holomorphic vector fields if $2$ is not an eigenvalue of $\Delta^g.$ Moreover, if $X$ is a Killing vector field, then $X$ is a Hamiltonian vector field.
		\end{enumerate}
	\end{cor*}
	We observe that in the K\"ahler setting, if $\rho^\nabla=\omega$ and $u$ is a function such that $\Delta^gu=2u$, then $u$ is a Killing potential. Hence, one can prove that the automorphism group is reductive and show the maximality of the isometry group in the automorphism group. However, this is not a priori true in the strictly almost-K\"ahler case (see Remark~\ref{remark2}). \\

	The paper is organized as follows: Section \ref{sec: prelim} provides a short recap of the required knowledge of almost-K\"ahler manifolds. Moreover, a technical result (Lemma \ref{lichnerowicz-AK}) is stated, while we dedicate an appendix to its proof.
	In Section \ref{sec1}, we define generalized almost-K\"ahler--Ricci solitons and discuss obstructions to their existence, with particular focus on the symplectic Fano case.
	Then, we perform our deformation argument in Section \ref{sec2}, and, in Section \ref{toric-section}, we specialize to the toric case.
	Finally, in Section \ref{sec-lie-algebra} we analyze the properties of the Lie algebra of holomorphic vector fields on manifolds admitting GeAKRS.
	
\section*{Acknowledgements}
	The second author is extremely grateful to Professor Mehdi Lejmi and the Math Department of the CUNY Graduate Center for their warm hospitality while conducting the research for this work. The authors are grateful to Eiji Inoue for his invaluable comments.
	
\section{Preliminaries}\label{sec: prelim}
	In the following section $(M,\omega)$ will be a symplectic manifold of dimension $2n$. 	
	The canonical Hermitian connection (a.k.a. the Chern connection)~\cite{MR66733,MR0165458,MR1456265} $\nabla$ on an almost-K\"ahler manifold is defined by
	\begin{equation*}
		\nabla_XY=D^g_XY-\frac{1}{2}J\left(D^g_XJ\right)Y,
	\end{equation*}
	where $X,Y$ are vector fields and $D^g$ is the Levi-Civita connection with respect to the metric $g$. The connection $\nabla$ is the unique connection preserving the almost-K\"ahler structure (i.e. $\nabla\omega = \nabla J = \nabla g = 0$) and whose torsion is the Nijenhuis tensor $N$. We denote by $R^\nabla$ the curvature of $\nabla$ and we use the convention $R^\nabla_{X,Y}=\nabla_{[X,Y]}-[\nabla_X,\nabla_Y]$. Then, the {\it first-Chern--Ricci form} $\rho^\nabla$ is defined by
	\begin{equation*}
		\rho^\nabla(X,Y):=\frac{1}{2}\sum_{i=1}^{2n}g\left( R^\nabla_{X,Y}e_i,Je_i\right),
	\end{equation*}
	where $\{e_1,e_2=Je_1,\cdots,e_{2n-1},e_{2n}=Je_{2n-1}\}$ is a $J$-adapted $g$-orthonormal local frame of the tangent bundle $TM.$
	The (real) $2$-form $\rho^\nabla$ is a de Rham representative of $2\pi c_1(M)$. However, $\rho^\nabla$ is not necessarily $J$-invariant, i.e., $\rho^{\nabla}$ is not necessarily a $(1, 1)$-form. The Chern scalar curvature $s^c$ is the trace of $\rho^\nabla$ with respect to $\omega$ i.e. 
	\begin{equation*}
		s^c\omega^n=n\,\,\rho^\nabla\wedge\omega^{n-1}.
	\end{equation*}
	Furthermore, we denote by $\Lambda_\omega$ the contraction by the symplectic form $\omega$ and by $$\delta^c:=J\delta^gJ^{-1}$$
	the twisted codifferential acting on $p$-forms, where $\delta^g$ is the codifferential with respect to the metric $g$. 
	The commutator of $\Lambda_\omega$ and the exterior derivative $d$ satisfies the K\"ahler identity (see for example~\cite{MR1637093})
	\begin{equation}\label{commutator}
		[\Lambda_\omega,d]=-\delta^c.
	\end{equation}
	We also define the twisted Laplace operator $\Delta^{c}$ acting on $p$-forms as $$\Delta^{c}:=\delta^cd^c+d^c\delta^c=J\Delta^gJ^{-1},$$ 
	where $d^c:=JdJ^{-1}$ is the twisted differential acting on $p$-forms and $\Delta^g=\delta^gd+d\delta^g$ is the Riemannian Laplacian with respect to the almost-K\"ahler metric $g$.
	The Hodge decomposition of a $p$-form $\psi$ with respect to $\Delta^{c}$ is given by: 
	\begin{equation}\label{hodge-twisted-decomposition}
		\psi=\psi_{H^c}+d^c\phi+\delta^c\tau,
	\end{equation}
	where $\psi_{H^c}$ is the $\Delta^c$-harmonic part, $\phi$ is a $(p-1)$-form and $\tau$ is a $(p+1)$-form. We call $d^c\phi$ the $d^c$-exact part of the
	Hodge decomposition of $\psi$ with respect to $\Delta^{c}$. 
	We remark that for an almost-K\"ahler metric $g$,  $\Delta^g-\Delta^c=[\Lambda_\omega,\left(dd^c+d^cd\right)]$~\cite{Gauduchon-book}, while $\Delta^g=\Delta^c$ when the metric is K\"ahler.
	Moreover, we have the following useful lemma on almost-K\"ahler manifolds (the proof is in the Appendix, see also a related expression for the Lichnerowicz operator in~\cite{MR4620280})
	
	\begin{lem}\label{lichnerowicz-AK}
		For any function $f$ on an almost-K\"ahler manifold $(M,\omega,J,g)$, we have
		\begin{equation*}
			\delta^g\left(\left(D^gd^cf\right)^{sym}_{J\cdot,\cdot}\right)=-\frac{1}{2}d(\Delta^g)f-\rho^\nabla(\left(d^cf\right)^{\sharp},\cdot),
		\end{equation*}
		where $\delta^g$ is the adjoint of the Levi-Civita connection $D^g$ with respect to $g,$ $\left(\cdot\right)^{sym}$ denotes the $g$-symmetric part, and
		$\sharp$ is the $g$-Riemannian dual.
	\end{lem}
	%
	
	As a consequence of Lemma~\ref{lichnerowicz-AK}, on an almost-K\"ahler manifold $(M,\omega,J,g)$, if the symplectic gradient $grad_\omega f=\left(d^cf\right)^\sharp$ is a Killing vector field, then~\cite{MR2747965}
	\begin{equation}\label{contraction-with-killing}
		-\frac{1}{2}d\Delta^gf=\rho^\nabla(grad_\omega f, \cdot).
	\end{equation}

\section{Generalized almost-K\"ahler Ricci solitons}\label{sec1}

	In the following, $(M,\omega)$ will be a symplectic manifold, and we denote by $Ham(M,\omega)$ the group of Hamiltonian symplectomorphisms of $(M,\omega)$.
	We fix a compact Lie subgroup $G$ (with Lie algebra $\mathfrak{g}$) of the group $Ham(M,\omega)$ and we denote by $AK^G_\omega$ the Fr\'echet manifold of $G$-invariant almost-complex structures compatible with $\omega$ on $M$.
	We define by $Ham^G(M,\omega)$ the normalizer of $G$ in $Ham(M,\omega)$, i.e., the space of Hamiltonian symplectomorphisms commuting with $G$. 
	Its Lie algebra $\mathfrak{ham}^G$ is the ideal generated by $\mathfrak{g}$ in $\mathfrak{ham}$.
	By taking the Hamiltonian potential, $\mathfrak{ham}^G$ can be identified with $C^\infty_G(M)/\mathbb{R}$, namely, the space of smooth $G$-invariant functions on $M$ up to additive constant.\\
	It is known that the space $AK_\omega$ comes equipped with a formal K\"ahler structure~\cite{MR1073369}, and the observations of Fujiki in the integrable case~\cite{MR1073369}, and of Donaldson in the general non-integrable case~\cite{MR1622931}, show that the natural action of $Ham(M,\omega)$ on $AK_\omega$ is Hamiltonian with moment map given by the Chern scalar curvature. 
	More precisely, the moment map $\mu:AK_\omega\rightarrow\mathfrak{ham}^*$ is
	$$ \mu(J)(Z) = -\int_M f s^c(J) \omega^n, $$
	where $Z=grad_\omega f$ and $f$ is normalised such that it has integral zero. 
	It turns out that it is also possible to give a more general formulation by fixing a connected compact subgroup of $Ham(M,\omega)$, which in our case is the compact group $G$.
	We present this following the formulation introduced by E. Inoue~\cite{MR4017922}.
	Start by fixing an element $\xi$ in the center of $\mathfrak{g}$. Note that $\xi$ defines a vector field on $M$ which we will also denote by $\xi$. Let $f_\xi$ be the potential defined by $$\omega(\xi, \cdot)=-df_\xi$$ and normalized such that $$\int_M f_\xi\,e^{-2f_\xi}\,\omega^n=0.$$
	We identify the Lie algebra $\mathfrak{ham}^G$ with the space of normalized functions in $C^\infty_G(M)$ 
	$$C^\infty_{G,\xi}(M,\omega)=\left\{u\in C^\infty_G(M)\,\middle|\, \int_Mu\,e^{-2f_\xi}\,\omega^n=0\right\} \simeq C^\infty_G(M)/\mathbb{R},$$
	which is equipped with the following inner product
	\begin{equation}\label{inner-product}
		\langle u,v\rangle_\xi:=  \int_Mu\,v\,e^{-2f_\xi}\,\omega^n.
	\end{equation}
	We use the potential $f_\xi$ to also twist the natural K\"ahler structure on $AK_\omega^G$.
	More precisely, we consider the following Riemannian metric on $AK_\omega^G$, modified by $\xi$ from the usual one:
	$$ (A,B)_\xi := \int_M \tr\left(AB\right) e^{-2f_\xi}\omega^n.$$
	The space $AK_{\omega}^G$ admits an almost-complex structure $\bm{J}$ which is compatible with this Riemannian metric. Extending the results of~\cite{MR1073369,MR1622931} (see also~\cite{MR3941493,MR3897025}), Inoue proved that that the natural action of $Ham^G(M,\omega)$ on $\left(AK^G_\omega, (\bm{J}\cdot,\cdot)_\xi\right)$ is Hamiltonian with moment map given by a {\em modified Chern scalar curvature}~\cite{MR2831983,MR4017922}.
	In details, the moment map is 
	\begin{eqnarray*}
		\mu_\xi:AK^G_\omega&\longrightarrow& (C^\infty_{G,\xi}(M,\omega))^{\ast},\\
		J&\mapsto& \langle 4s^c_\xi(J),\cdot\rangle_\xi,
	\end{eqnarray*}
	where for a given $J\in AK^G_\omega$, $s^c$ is the Chern scalar curvature of $(\omega,J)$, $\Delta^g$ is the Riemannian Laplacian of the metric induced by $(\omega,J)$, and
	$$s^c_\xi(J):=s^c(J)-n-2\Delta^gf_\xi+2f_\xi-2|df_\xi|_g^2$$  (note that in~\cite{MR4017922} $G$ is a torus but the assumption that $\xi$ is in the center of $\mathfrak{g}$ is enough to ensure that $\mu_{\xi}$ is a moment map).

	\begin{defn}
		We say an almost-K\"ahler metric $g$ is a {\it{generalized almost-K\"ahler--Ricci soliton}} (GeAKRS for short), if, for some $\xi$, the almost-complex structure inducing $g$ is a zero of the moment map $\mu_\xi$. That is, $g$ is an almost-K\"ahler metric satisfying $$s^c_\xi\equiv0.$$ 
	\end{defn}

	If $\xi$ is a trivial vector field then a GeAKRS is an almost-K\"ahler metric of constant Chern scalar curvature, while in the (compact) K\"ahler Fano case, a GeAKRS corresponds to a K\"ahler--Ricci soliton~\cite{MR4017922}. 
	We also remark that in the K\"ahler setting, D. Guan~\cite{MR1327154, MR2366370} introduced generalized Quasi-Einstein metrics. 
	They correspond to K\"ahler--Ricci solitons in the Fano case, and are obstructions to the existence of constant scalar curvature K\"ahler metrics (see also~\cite{MR2805600,MR3078263}). However, GeAKRS do not a priori coincide with Quasi-Einstein metrics in the non-Fano K\"ahler case.

	As a natural consequence of the moment map setup we obtain a symplectic Futaki invariant.
	In particular, once we identify $\mathfrak{g}$ with a subset of $C^\infty_{G,\xi}(M,\omega)\simeq\mathfrak{ham}^G$, then for any $u\in \mathfrak{g}$ the integral 
	\begin{equation}\label{futaki integral}
		\int_Ms^c_\xi(J) \,u\,e^{-2f_\xi}\omega^n,
	\end{equation}
	is independent of $J$ on any connected component of $AK_\omega^G$~\cite{MR4017922}. 
	Hence, we can associate to any connected component of $AK_\omega^G$ a linear map $\mathcal{F}^{\xi}_G:\mathfrak{g}\rightarrow\mathbb{R}$ defined as
	\begin{equation*}
		\mathcal{F}^{\xi}_G(u)=\int_Ms^c_\xi(J) \,u\,e^{-2f_\xi}\omega^n.
	\end{equation*}
	Of course, a necessary condition that the chosen connected component of $AK^G_\omega$ contains a GeAKRS is that $\mathcal{F}^{\xi}_G$ is identically zero on $\mathfrak{g}.$
	
	The fact that the integral in \eqref{futaki integral} is independent of $J$ can be rephrased by saying that changing the almost-complex structure, $s^c_\xi$ varies orthogonally to $\mathfrak{g}$ with respect to the inner product in \eqref{inner-product}.
	Therefore, if we denote by $\Pi_{\omega}$ the orthogonal projection on $\mathfrak{g}$ with respect to the inner product~\eqref{inner-product}, then $\Pi_{\omega}s^c_\xi(J)$ is independent of $J$ on any connected component of  $AK^G_\omega$.
	We state this in the following
	\begin{prop}\label{prop: const proj}
		The map $AK_\omega^G \rightarrow C^\infty(M)$ associating to any $G$-invariant complex structure $J$ the function $\Pi_{\omega}s^c_\xi(J)$ is constant on any connected component of $AK^G_\omega$.
	\end{prop}
	We can thus define the {\em extremal vector field} relative to $G$,  denoted $Z^G_\omega\in\mathfrak{g}$, to be the symplectic gradient of $\Pi_{\omega}s^c_\xi$.
	\begin{rem}
		It holds that $Z^G_{\tilde{\omega}}=Z^G_{\omega}$, for any $G$-invariant symplectic form $\tilde{\omega}$ isotopic to $\omega.$
	\end{rem}

	A compact symplectic manifold $(M,\omega)$ is called \textit{monotone} if $2\pi c_1(M)=\lambda [\omega]$, for some constant $\lambda$.
	We prove that on a compact symplectic monotone manifold $(M,\omega)$, if $J\in AK^G_\omega$ induces a GeAKRS, then $\lambda$ has to be positive unless $\xi$ is trivial.
	\begin{prop}\label{positivity}
		Let $(M,\omega)$ be a compact symplectic monotone manifold. Suppose that $J\in AK^G_\omega$ induces a GeAKRS for some non-trivial vector field $\xi$. Then $\lambda$ is positive.
	\end{prop}
	\begin{proof}
		From the monotone hypothesis, we have  
		\begin{equation}\label{Fano}
			\rho^\nabla-\lambda\omega=da,
		\end{equation}
		for some $G$-invariant $1$-form $a$. Now, contracting Equation~(\ref{Fano}) with $\xi$ and
		using the identity~(\ref{contraction-with-killing}) and Cartan formula we get:
		$$-\frac{1}{2}d\Delta^g f_\xi+\lambda df_\xi=-d\left(a(\xi)\right).$$
		Hence, for some constant $c$,
		\begin{equation}\label{contraction}
			-\frac{1}{2}\Delta^g f_\xi+\lambda f_\xi=-a(\xi)+c.
		\end{equation}
		Using the induced metric on one-forms, one can rewrite the term $a(\xi)$ as follows:
		\begin{equation}\label{a(xi)}
		a(\xi) = -g(Ja, df_{\xi}).
		\end{equation}
		So multiplying Equation (\ref{contraction}) by $e^{-2f_\xi}$ and integrating over $M$, we get
		\begin{eqnarray}\label{integral}
			-\frac{1}{2}\int_M\Delta^g f_\xi\,e^{-2f_\xi}\omega^n&=&\int_Mg(Ja,df_\xi)\,e^{-2f_\xi}\omega^n+c\int_Me^{-2f_\xi}\omega^n,\nonumber\\
		 &=&-\frac{1}{2}\int_Mg\left(Ja,d(e^{-2f_\xi})\right)\omega^n+c\int_Me^{-2f_\xi}\omega^n,\nonumber\\
		 &=&-\frac{1}{2}\int_M\left(\delta^gJa\right)\,e^{-2f_\xi}\omega^n+c\int_Me^{-2f_\xi}\omega^n.
		\end{eqnarray}
		On the other hand, by contracting~(\ref{Fano}) with $\omega$ and using identity~(\ref{commutator}), we have
		\begin{equation}\label{scalar}
		s^c-\lambda n=\delta^g Ja.
		\end{equation}
		Therefore, Equation~(\ref{integral}) becomes
		\begin{equation}\label{integral_1}
			-\frac{1}{2}\int_M\Delta^g f_\xi\,e^{-2f_\xi}\omega^n=-\frac{1}{2}\int_M\left(s^c-\lambda n\right)\,e^{-2f_\xi}\omega^n+c\int_Me^{-2f_\xi}\omega^n.
		\end{equation}
		Moreover, by hypothesis, $s^c_\xi=s^c-n-2\Delta^gf_\xi+2f_\xi-2|df_\xi|_g^2=0$ so multiplying by $e^{-2f_\xi}$ and integrating we get
		$$\int_M\Delta^g f_\xi\,e^{-2f_\xi}\omega^n=\int_M\left(s^c- n\right)\,e^{-2f_\xi}\omega^n.$$
		Hence, Equation~(\ref{integral_1}) is equivalent to
		$$ -\frac{1}{2}\int_M\left(s^c- n\right)\,e^{-2f_\xi}\omega^n=-\frac{1}{2}\int_M\left(s^c-\lambda n\right)\,e^{-2f_\xi}\omega^n+c\int_Me^{-2f_\xi}\omega^n, $$
		that is
		$$ 0=\left(\frac{n(\lambda-1)}{2}+c \right)\int_Me^{-2f_\xi}\omega^n. $$

		Consequently, the constant $c=-\frac{n(\lambda-1)}{2}.$ 
		Then, from Equations~(\ref{contraction}) and~(\ref{a(xi)}), we deduce that at a maximum point $p$ of $f_\xi$,
		$$-\lambda f_\xi (p)\leq \frac{n(\lambda-1)}{2},$$ 
		but the normalization $\int_M f_\xi\,e^{-2f_\xi}\,\omega^n=0$ implies that if $\lambda\leq 0$ the constant $c=0$ and $f_\xi\equiv 0.$
	\end{proof}

	We recall that the class of ${\rho}^\nabla$ in de Rham cohomology does not depend on $J\in AK^G_\omega$, see e.g. (9.5.14) of~\cite{Gauduchon-book}.
	Moreover, if two symplectic forms $\omega$ and $\tilde{\omega}$ are compatible with $J$ and satisfy $\tilde{\omega}^n=e^F\omega^n,$ for some real-valued function $F$ then  (see for instance~\cite{MR4184828}) $$\tilde{\rho}^\nabla={\rho}^\nabla-\frac{1}{2}dd^cF,$$
	where $\tilde{\rho}^\nabla$ is the first-Chern--Ricci form of $(\tilde{\omega},J)$. 
	Hence, for symplectic monotone manifolds, we can reduce to the case where $\lambda$ is $1,0$, or $-1$. 
	With this assumption, we call a compact symplectic manifold $(M,\omega)$
	satisfying $2\pi c_1(M)= [\omega]$ a {\it{symplectic Fano manifold}}. 
	It then holds that for any $J\in AK^G_\omega$, we have $\rho^\nabla-\omega=da,$ for some $G$-invariant $1$-form $a$.
	Now we can characterize GeAKRS on compact symplectic Fano manifolds. 

	\begin{prop}\label{GeAKRS-exact-dc}
		Let $(M,\omega)$ be a compact symplectic Fano manifold of real dimension $2n > 2$. Then the almost-K\"ahler metric $g$ induced by $(\omega,J)$ is a GeAKRS
		if and only if the $d^c$-exact part of the Hodge decomposition of the $1$-form $a$ with the respect to the twisted Laplacian $\tilde{\Delta}^c=J\Delta^{\tilde{g}}J^{-1}$ of the conformal metric $\tilde{g}=e^{-\frac{2f_\xi}{n-1}}g$ is $-d^c f_\xi$.
	\end{prop}
	\begin{proof}
		We suppose that the metric $g$ induced by  $(\omega,J)$ is a GeAKRS. Then $\rho^\nabla-\omega=da$ for some $G$-invariant $1$-form $a$.
		It follows from Equation~(\ref{contraction}), with $\lambda = 1$, and Equation~(\ref{a(xi)}), that
		\begin{equation}\label{f-expression}
			-\frac{1}{2}\Delta^g f_\xi+ f_\xi=g(Ja, df_{\xi}).
		\end{equation}
		On the other hand, since the metric is a GeAKRS, we have
		$$\delta^gJa-2\Delta^gf_\xi+2f_\xi-2|df_\xi|_g^2=0$$
		where we have used~(\ref{scalar}), again with $\lambda = 1$. In particular $$f_\xi=-\frac{1}{2}\delta^gJa+\Delta^gf_\xi+|df_\xi|_g^2.$$
		Substituting into~(\ref{f-expression}) we get
		\begin{equation}\label{conformal-equation}
			\Delta^gf_\xi+2|df_\xi|_g^2=\delta^gJa+2g(Ja,df_\xi).
		\end{equation}
		Now, consider the metric $\tilde{g}=e^{-\frac{2f_\xi}{n-1}}g$ and denote by ${\delta}^{\tilde{g}}$ the codifferential with respect to the metric $\tilde{g}$.
		Then, Equation~(\ref{conformal-equation}) becomes
		\begin{equation}\label{conformal-expression}
			{\delta}^{\tilde{g}}df_{\xi}={\delta}^{\tilde{g}}Ja.
		\end{equation}
		Suppose that the Hodge decomposition of the $1$-form $a$ with respect to $\tilde{\Delta}^c=J\Delta^{\tilde{g}}J^{-1}$ is
		$$a={a}_{\tilde{H}^c}+d^ch+\tilde{\delta}^{c}\psi,$$
		where ${a}_{\tilde{H}^c}$ is $\tilde{\Delta}^c$-harmonic part, $h$ is a function, and $\psi$ is a $2$-form. Then, Equation~(\ref{conformal-expression})
		can be expressed as
		$$\Delta^{\tilde{g}}f_\xi=-\Delta^{\tilde{g}}h.$$
		Since $M$ is compact, it follows that $d^ch=-d^cf_\xi$. 
		
		For the converse, note that the hypothesis is equivalent to Equation~(\ref{conformal-expression}), and hence to  Equation~(\ref{conformal-equation}). Using Equations~(\ref{f-expression}) and~(\ref{scalar}), we obtain $s^c_{\xi} = 0$, so the almost-K\"ahler metric $g$ is a GeAKRS.
	\end{proof}
	We remark that the choice of the $G$-invariant $1$-form ${a}$ satisfying $\rho^\nabla-\omega=d{a}$ is irrelevant. Indeed, another choice of such $G$-invariant $1$-form ${a}^\prime$ satisfies $d{a}^\prime=da$. Contracting by $\omega$, we deduce using Equation~(\ref{commutator}) that ${a}$ and ${a}^\prime$ have the same $d^c$-exact part with respect to ${\Delta}^c=J\Delta^{{g}}J^{-1}$. Hence, $${\delta}^{\tilde{g}}Ja-2g(Ja,df_\xi)={\delta}^{\tilde{g}}J{a}^\prime-2g(J{a}^\prime,df_\xi),$$
where $\tilde{g}=e^{-\frac{2f_\xi}{n-1}}g$. Moreover, it follows from Cartan formula and the $G$-invariance of $a$ and ${a}^\prime$ that $$d\left(g(J(a-{a}^\prime,df_\xi)\right)=0.$$
It follows that ${\delta}^{\tilde{g}}Ja={\delta}^{\tilde{g}}J{a}^\prime$ and so $a$ and ${a}^\prime$ have the same $d^c$-exact with respect to $\tilde{\Delta}^c=J\Delta^{\tilde{g}}J^{-1}$.
	
	We observe that when the metric $g$ is K\"ahler, the $1$-form $a$ is $d^c$-exact and one obtain a K\"ahler-Ricci soliton. Therefore, Proposition~\ref{GeAKRS-exact-dc} can be seen as a generalization of~\cite[Proposition 3.2]{MR4017922} to the almost-K\"ahler setting.
	We now prove that on a compact symplectic Fano manifold, the existence of GeAKRS in $AK^G_\omega$ with non-trivial $\xi$
	is an obstruction to the existence of an almost-K\"ahler metric of constant Chern scalar curvature in $AK^G_\omega.$ This extends to the almost-K\"ahler setting the fact that K\"ahler--Ricci solitons are obstructions to the existence of K\"ahler--Einstein metrics (see for instance~\cite[Theorem 1]{MR1327154}). We remark here that, on a symplectic Fano manifold $(M,\omega)$, an almost-K\"ahler metric of constant Chern scalar curvature induced by $(\omega,J)$ satisfies $\rho^\nabla=\omega$ if and only if $\rho^\nabla$ is $J$-invariant. 

	\begin{thm}~\label{obstruction-constant-chern}
		Let $(M,\omega)$ be a compact symplectic Fano manifold. Suppose that there exists $J\in AK^G_\omega$ which induces a GeAKRS with respect to a vector field $\xi$ and $\tilde{J}\in AK^G_\omega$ which induces an almost-K\"ahler metric of constant Chern scalar curvature. Then $\xi\equiv 0.$
	\end{thm}
	\begin{proof}
		Since $J$ induces a GeAKRS $g$, the corresponding potential $f_{\xi}$ satisfies Equation~(\ref{conformal-equation}). Integrating over $M$ yields
		\begin{equation}
			\int_M|df_\xi|_g^2\,\omega^n= \int_Mg(Ja, df_{\xi})\,\omega^n = \int_M f_\xi\,\left(\delta^gJa\right)\omega^n=\int_M f_\xi\,\left(s^c-n\right)\omega^n\label{projection}.
		\end{equation}
where the last equality uses Equation~(\ref{scalar}) with $\lambda = 1$.
		

		Now we denote by $\mathfrak{g}_{\omega}$ the finite-dimensional space of normalized smooth functions which are Hamiltonians of $Lie(G).$ Then the $L^2$-orthogonal projection of the Chern scalar curvature $s^c$ to $\mathfrak{g}_{\omega}$ is independent of $J\in AK^G_\omega$~\cite[Lemma 3.4]{MR2747965}. It follows that if $(\omega,\tilde{J})$ induces an almost-K\"ahler metric of constant Chern scalar curvature for some $\tilde{J}\in AK^G_\omega$ then the $L^2$-orthogonal projection of $s^c-n$ to $\mathfrak{g}_{\omega}$ is zero. Hence, the right-hand side of Equation~(\ref{projection}) is zero and so $f_\xi\equiv 0.$
	\end{proof}

	\begin{ex}
		It is known that there exist (toric) K\"ahler--Ricci solitons on the blow up of the complex projective space at one point $\mathrm{C}P^2\sharp\bar{\mathrm{C}P}^2$ and at two points $\mathrm{C}P^2\sharp 2\bar{\mathrm{C}P}^2$~\cite{MR1145263,MR1417944,MR2084775}. Theorem~\ref{obstruction-constant-chern} implies that there are no (toric) almost-K\"ahler metrics of constant Chern scalar curvature in the symplectic Fano case on such manifolds. 
	\end{ex}

	We end this section by extending a result obtained in the K\"ahler case in~\cite{MR4017922} to our more general setting.
	\begin{prop}
		Let $(M,\omega)$ be a compact symplectic Fano manifold of real dimension $2n > 2$ with $\rho^{\nabla} - \omega = da$ for a fixed $J \in AK_{\omega}^G$ and a $G$-invariant $1$-form a. Then
		\begin{equation*}
			\mathcal{F}^{\xi}_G(u)=-\int_M \Delta^{\tilde{g}} (h+f_\xi)\,u\,e^{-2f_\xi}\omega^n,
		\end{equation*}
		where $\tilde{g}=e^{-\frac{2f_\xi}{n-1}}g$, and $h\in C^\infty_{G,\xi}(M,\omega)$ is the function
		defined such that $d^ch$ is the $d^c$-part in the Hodge decomposition~(\ref{hodge-twisted-decomposition}) of the $1$-form $a$ with the respect to the twisted Laplacian $\tilde{\Delta}^c=J\Delta^{\tilde{g}}J^{-1}$.
	\end{prop}
	\begin{proof}
		We first recall that the $1$-form $a$ decomposes as $a=a_{\tilde{H}^c}+d^ch+\tilde{\delta}^c\tau$ with respect to the twisted Laplacian $\tilde{\Delta}^c$, for some $G$-invariant function $h$ and a $2$-form $\tau.$
		From Equation~(\ref{scalar}) with $\lambda = 1$, we have $s^c - n =\delta^gJa$, and by combining Equations~(\ref{contraction}) and (\ref{a(xi)}), we have $-\Delta^gf_\xi+2f_\xi=2g(Ja,df_\xi)+c$. We therefore obtain the following chain of equalities
		\begin{eqnarray*}
			\mathcal{F}^{\xi}_G(u)&=&\int_Ms^c_\xi \,u\,e^{-2f_\xi}\omega^n,\\
			&=&\int_M\left(s^c-n-2\Delta^gf_{\xi}+2f_\xi-2|df_\xi|^2\right)u\,e^{-2f_\xi}\omega^n,\\
			&=&\int_M\left(\delta^gJa-2\Delta^gf_{\xi}+2f_\xi-2|df_\xi|^2\right)u\,e^{-2f_\xi}\omega^n,\\
			&=&\int_M\left(\delta^gJa+2g(Ja,df_\xi)+c-\Delta^gf_\xi -2|df_\xi|^2\right)u\,e^{-2f_\xi}\omega^n,\\
			&=&\int_M\left(\delta^gJa+2g(Ja,df_\xi)-\Delta^gf_\xi -2|df_\xi|^2\right)u\,e^{-2f_\xi}\omega^n,\\
			&=&\int_M\left(\delta^{\tilde{g}}Ja-\Delta^{\tilde{g}}f_\xi \right)u\,e^{-2f_\xi}\omega^n,\\
			&=-&\int_M\Delta^{\tilde{g}}\left(h+f_\xi \right)u\,e^{-2f_\xi}\omega^n,
		\end{eqnarray*}
		where we have used the normalization $\int_Mu\,e^{-2f_\xi}\omega^n=0.$ 
	\end{proof}

\section{Deforming K\"ahler--Ricci solitons to generalized almost-K\"ahler--Ricci solitons}\label{sec2}
 
	In this section, we suppose that $(M,\omega)$ is a $4$-dimensional compact symplectic Fano manifold and $T$ is a maximal torus in $Ham(M,\omega)$.
	Suppose that $J_0\in  AK^T_\omega$ induces a K\"ahler-Ricci soliton with respect to a vector field $\xi\in \mathfrak{t}=Lie(T)$.
	We denote by $\mathfrak{t}_{\omega}\subset C^\infty_{T,\xi}(M,\omega)$ the finite-dimensional space of normalized smooth functions which are Hamiltonians of $\mathfrak{t}.$
	Let $J_t\in AK^T_\omega$ be a family of $T$-invariant, almost-complex structures compatible with $\omega$ such that $h^-_{J_t}$, the dimension of $g_t$-harmonic $J_t$-anti-invariant $2$-forms, satisfies $h^-_{J_t}=b^+(M)-1$, for small $t$. Here, $g_t$ is the almost-K\"ahler metric induced by $J_t.$ We consider the following $T$-invariant almost-K\"ahler deformations~\cite{MR3331165}
	\begin{equation*}
		\omega_{t,\phi}=\omega+d\mathbb{G}_tJ_td\Delta^{g_t}\phi,
	\end{equation*} 
	where $\mathbb{G}_t$ is the Green operator associated to the Riemannian Laplacian $\Delta^{g_t}$ and $\phi$ belongs to the space 
	$C^{\infty,\perp}_{T,\xi}(M,\omega)\subset C^{\infty}_{T,\xi}(M,\omega)$ orthogonal to $\mathfrak{t}_{\omega}$ with respect to the inner product~(\ref{inner-product}). We consider $W^{p,k}_{T,\xi,\perp}$, the completion of $C^{\infty,\perp}_{T,\xi}(M,\omega)$ with respect to the Sobolev norm involving derivatives up to order $k.$ Here $p$ and $k$ are chosen so that all coefficients are continuous.
	In the spirit of LeBrun--Simanca~\cite{MR1274118} (see also~\cite{MR2661166,MR3663320,MR3897025}), we define the map
	\begin{equation*}
		\begin{array}{lrcl} \Phi : &\mathcal{U}& \longrightarrow &
		\mathbb{R}\times  W^{p,k}_{T,\xi,\perp}  \\
		    & (t,\phi)& \longmapsto
		    &\left(t,\left(Id-\Pi_{\omega}\right)\left(Id-\Pi_{{\omega_{t,\phi}}}\right) \left(s^c_{\xi,t,\phi}\right)\right),
		\end{array}
	\end{equation*} 
	where $\mathcal{U}$ is an open neigbourhood of $(0,0)\in\mathbb{R}\times W^{p,k+4}_{T,\xi,\perp} $, and $\Pi_{\omega}$ (resp. $\Pi_{{\omega_{t,\phi}}}$) is the orthogonal projection to $\mathfrak{t}_{\omega}$ (resp. $\mathfrak{t}_{{\omega_{t,\phi}}}$) with respect to the inner product~(\ref{inner-product}), and $s^c_{\xi,t,\phi}$ is the modified Chern scalar curvature of $(\omega_{t,\phi},J_t)$. By possibly replacing $\mathcal{U}$ with a smaller open set, we can suppose that $\Phi(t,\phi)=(t,0)$ if and only if $s^c_{\xi,t,\phi}\in \mathfrak{t}_{{\omega_{t,\phi}}}$~\cite{MR1274118}. Since $s^c_{\xi,0,0}\equiv 0$, the extremal vector field $Z^T_\omega$ vanishes and by Proposition~\ref{prop: const proj}, we obtain
	that $\Phi(t,\phi)=(t,0)$ if and only if $s^c_{\xi,t,\phi}\equiv 0$ so that $(\omega_{t,\phi},J_t)$ is a GeAKRS.
	We now compute the differential of $\Phi$ with respect to $\phi$ at $(0,0)$. 
	We have the following
	
	\begin{prop}\label{differential-second}
		For a K\"ahler variation $\dot{\omega}=dd^c\dot{\phi}$, we have 
		\begin{eqnarray}
			\dot{s^c}&=&-\frac{1}{2}\left(\Delta^g\right)^2\dot{\phi}-g(\rho^\nabla,dd^c\dot{\phi}),\label{chern-variation}\\
			\dot{\left(\Delta^gf_\xi\right)}&=&g(d^c\Delta^g\dot{\phi},\xi),\label{laplace-variation}\\
			\dot{f_\xi}&=&g(d^c\dot{\phi},\xi),\label{potential-variation}\\
			\dot{|\xi|^2_g}&=&g\left(d^c\left(g(d^c\dot{\phi},\xi)\right),\xi  \right).\label{norm-variation}
		\end{eqnarray}
		It follows that the variation of the modified Chern scalar curvature is
		\begin{eqnarray}\label{modified-variation}
			\frac{\partial s^c_{\xi,0,\phi}}{\partial \phi}\Bigg|_{(0,0)}(\dot{\phi})&=&-\frac{1}{2}\left(\Delta^g\right)^2\dot{\phi}-\Delta^g\left(d^c\dot{\phi},\xi\right)+\Delta^g\dot{\phi}+2g(d^c\dot{\phi},\xi)\\
			&-&g(d^c\Delta^g\dot{\phi},\xi)-2g\left(d^c\left(g(d^c\dot{\phi},\xi)\right),\xi\right).\nonumber
		\end{eqnarray}
		Here $g$ is the K\"ahler metric induced by $(\omega,J_0)$ such that $s^c_{\xi,0,0}\equiv 0$ (i.e. $g$ is a K\"ahler--Ricci soliton).
	\end{prop}
	\begin{proof}
		First of all, Equation~(\ref{chern-variation}) follows from~\cite{MR1274118}.\\
		To establish Equation~(\ref{laplace-variation}), we have
		$\dot{\rho^\nabla}=\frac{1}{2}dd^c\Delta^g\dot{\phi}$~\cite{MR1274118} and so it follows from Equation
		~(\ref{contraction-with-killing}), that
		\begin{eqnarray*}
			-\frac{1}{2}d\left(\dot{\left(\Delta^gf_\xi\right)}\right)=\dot{\rho^\nabla}(\xi, \cdot) =\frac{1}{2}\left(dd^c\Delta^g\dot{\phi}\right)(\xi, \cdot) =-\frac{1}{2}d\left(g\left(d^c\Delta^g\dot{\phi},\xi\right)\right),
		\end{eqnarray*}
		hence
		$$\dot{\left(\Delta^gf_\xi\right)}=g\left(d^c\Delta^g\dot{\phi},\xi\right)+c_1,$$
		where $c_1$ is a constant. We claim that $c_1$ is zero, indeed, $\int_M\,(\Delta^gf_\xi)\,\omega^2=0$ and hence 
		\begin{eqnarray*}
			0&=&\int_M\dot{(\Delta^gf_\xi)}\,\omega^2+\int_M(\Delta^gf_\xi)\,\dot{\omega^2},\\
			&=&\int_M\dot{(\Delta^gf_\xi)}\,\omega^2-\int_M(\Delta^gf_\xi)(\Delta^g\dot{\phi})\,{\omega^2},\\
			&=&\int_M\dot{(\Delta^gf_\xi)}\,\omega^2-\int_Mg\left(d^c\Delta^g\dot{\phi},\xi\right){\omega^2}.
		\end{eqnarray*}
		Equation~(\ref{potential-variation}) follows from $\dot{\omega}(\xi, \cdot)=-d\dot{f_\xi}$ and so $\left(dd^c\dot{\phi}\right)(\xi, \cdot)=-d\dot{f_\xi}$. Using
		Cartan formula, we get that $g(d^c\dot{\phi},\xi)=\dot{f_\xi}+c_2,$ where $c_2$ is a constant. 
		We claim that $c_2$ is zero. Indeed, if we vary the integral $\int_Mf_\xi\,\omega^2$ we get that $$0=\int_M\dot{f_\xi}\,\omega^2+\int_Mf_\xi\,\dot{\omega^2}=\int_M\dot{f_\xi}\,\omega^2-\int_Mf_\xi\,\Delta^g\dot{\phi}\,{\omega^2}=\int_M\dot{f_\xi}\,\omega^2-\int_Mg(d^c\dot{\phi},\xi)\,{\omega^2}.$$  
		Equation~(\ref{norm-variation}) follows from Equation~(\ref{potential-variation}).
		
		To get the variation formula of the modified Chern scalar curvature $s^c_{\xi,0,\phi}$, we first combine Equations~(\ref{chern-variation}),
		~(\ref{laplace-variation}),~(\ref{potential-variation}), and~(\ref{norm-variation}). Then, we compute the term $g(\rho^\nabla,dd^c\dot{\phi})$. Here, we use the fact that $g$ is a K\"ahler--Ricci soliton. Indeed, $\rho^\nabla-\omega=-dd^cf_\xi$ so 
		\begin{eqnarray}
			g(\rho^\nabla,dd^c\dot{\phi})&=&g(\omega,dd^c\dot{\phi})-g(dd^cf_\xi,dd^c\dot{\phi}),\nonumber\\
			&=&-\Delta^g\dot{\phi}-g(d\xi,dd^c\dot{\phi}),\nonumber\\
			&=&-\Delta^g\dot{\phi}-g(\xi,\Delta^gd^c\dot{\phi})+\Delta^g\left(g(d^c\dot{\phi},\xi)\right),\nonumber\\
			&=&-\Delta^g\dot{\phi}-g(\xi,d^c\Delta^g\dot{\phi})+\Delta^g\left(g(d^c\dot{\phi},\xi)\right),\label{chern-ricci-term}
		\end{eqnarray}
		where we use the adjunction formula in the last line and the fact that $\Delta^gd^c=d^c\Delta^g$ in the K\"ahler setting. Finally, Equation~(\ref{modified-variation}) follows from~(\ref{chern-ricci-term}).
	\end{proof}

	\begin{prop}
		The differential of $\Phi$ at $(0,0)$ is given by
		\begin{equation}
			\left({\bold{T}}_{(0,0)}\Phi\right)(t,\dot{\phi})=\left(t,t\left(\frac{1}{2}\delta^{g_0}\delta^{g_0}\dot{g}+2\delta^{g_0}\dot{J}d^cf_\xi-2\dot{g}(\xi,\xi)\right)+\mathbb{L}_\xi\left(\Delta^{g_0}\dot{\phi}+2\,g_0\left(d^c\dot{\phi},\xi\right)\right)\right),
		\end{equation}
		where $\dot{g}=\frac{d}{dt}|_{t=0}g_t,\dot{J}=\frac{d}{dt}|_{t=0}J_t$, and $\mathbb{L}_\xi$ is the linear elliptic self-adjoint (with respect to the inner product~(\ref{inner-product})) operator defined by
		$$\mathbb{L}_\xi(\phi)=-\frac{1}{2}\Delta^{g_0}\phi+\phi-g_0(d^c\phi,\xi).$$
	\end{prop}
	\begin{proof}
		Using Proposition~\ref{prop: const proj}, the projection $\Pi_{{\omega_{t,0}}} \left(s^c_{\xi,t,0}\right)$ is independent of $t$. Moreover,
		the variation of the Chern scalar curvature $s^c_{t,0}$ of $(\omega,J_t)$ is given by $\frac{1}{2}\,\delta^{g_0}\delta^{g_0}\dot{g}$ (see~\cite{MR1274118}). Using the fact that $\Delta^{g_t}f_\xi\,\omega^2=-2\,dJ_tdf_\xi\wedge \omega$, together with Equation~(\ref{commutator}), we deduce that
		$\frac{d}{dt}|_{t=0}\left(\Delta^{g_t}\right)(f_\xi)=\delta^{g_0}\dot{J}d^cf_\xi$. Then the differential of $\Phi$ with respect to $\phi$ follows from Proposition~\ref{differential-second}.
	\end{proof}

	Recall that on a compact K\"ahler manifold, a symplectic gradient $grad_\omega \psi$ is holomorphic if and only if $-\frac{1}{2}d\Delta^{g}\psi=\rho^\nabla(grad_\omega \psi, \cdot)$~\cite{MR0124009} (see also~Lemma \ref{lichnerowicz-AK}). Since $(\omega,J_0)$ is a K\"ahler--Ricci soliton, we have $\rho^\nabla-\omega=-dd^cf_\xi,$ so for a $T$-invariant function $\psi$, we see that $grad_\omega \psi$ is holomorphic if and only if $\mathbb{L}_\xi(\psi)=0$~\cite[Lemma 2.2]{MR1768112}. By the maximality of the torus $T$, it follows that ${\psi}\in \operatorname{ker}\left(\mathbb{L}_{\xi}\right)$ if and only if $\psi\in \mathfrak{t}_{\omega}.$ Now, if $\left({\bold{T}}_{(0,0)}\Phi\right)(t,\dot{\phi})=(0,0)$, then $\mathbb{L}_\xi\left(\Delta^{g_0}\dot{\phi}+2\,g_0\left(d^c\dot{\phi},\xi\right)\right)=0$ and therefore $$\Delta^{g_0}\dot{\phi}+2\,g_0\left(d^c\dot{\phi},\xi\right)\in \mathfrak{t}_{\omega}.$$
	However, $\dot{\phi}$ is orthogonal to $ \mathfrak{t}_{\omega}$ with respect to the inner product~(\ref{inner-product}). Hence, we have
	\begin{eqnarray*}
		0&=&\int_M\dot{\phi}\,\left(\Delta^{g_0}\dot{\phi}+2\,g_0\left(d^c\dot{\phi},\xi\right)\right)\,e^{-2f_\xi}\omega^2,\\
		&=&\int_M \|d\dot{\phi}\|^2_{g_0}+\dot{\phi}\,g_0\left(d\dot{\phi},d\left(e^{-2f_\xi}\right)\right)+2\,\dot{\phi}\,g_0\left(d^c\dot{\phi},\xi\right)\,e^{-2f_\xi}\,\omega^2,\\
		&=&\int_M \|d\dot{\phi}\|^2_{g_0}-2\,\dot{\phi}\,g_0\left(d^c\dot{\phi},\xi\right)\,e^{-2f_\xi}+2\dot{\phi}\,g_0\left(d^c\dot{\phi},\xi\right)\,e^{-2f_\xi}\,\omega^2,\\
		&=&\int_M \|d\dot{\phi}\|^2_{g_0}\omega^2.
	\end{eqnarray*}
	Thus, if $\left({\bold{T}}_{(0,0)}\Phi\right)(t,\dot{\phi})=(0,0)$ then $\dot{\phi}\equiv 0.$ Using standard elliptic theory, we deduce that ${\bold{T}}_{(0,0)}\Phi$
	is an isomorphism from $\mathbb{R}\oplus W^{p,k+4}_{T,\xi,\perp}$ to $\mathbb{R}\oplus W^{p,k}_{T,\xi,\perp}$. It follows from the inverse function theorem for Banach spaces that for a small $t$ there exists $\phi_t$ such that $s^c_{\xi,t,\phi_t}\equiv 0.$ 

	So far we achieved the existence of solutions $\phi_t$ of the GeAKRS equation in $W^{p,k+4}_{T,\xi,\perp}$.
	The Sobolev embedding for $k$ large enough ensures that we find an a priori $\mathcal{C}^4$ bound on the solutions;
	consequently, the metrics $\omega_{t,\phi_t}$ are at least $\mathcal{C}^2$. 
	Moreover, the GeAKRS equation satisfied by $\phi_t$ can be written as
	$$ s^c(\omega_{t,\phi_t}) = n+2 \Delta_{\omega_{t,\phi_t}}f_\xi -2f_\xi +2|df_\xi|_{\omega_{t,\phi_t}}^2. $$
	Notice that, since $f_\xi$ is a fixed smooth function, the right-hand side in the above equation has the same regularity as $\omega_{t,\phi_t}$.
	Therefore $s^c(\omega_{t,\phi_t})$ has the same regularity as $\omega_{t,\phi_t}$.
	Now, the Chern scalar curvature has the following expression
	$$ s^c(\omega_{t,\phi_t}) = \Lambda_{\omega_{t,\phi_t}}\rho^\nabla(\omega_{t,\phi_t}) =
	\Lambda_{\omega_{t,\phi_t}}\left(\rho^\nabla(\omega)- \frac{1}{2}dJ_td \log\left(\frac{v_{\omega_{t,\phi_t}}}{v_{\omega}}\right)\right),$$
	therefore,
	$$ s^c(\omega_{t,\phi_t}) - \Lambda_{\omega_{t,\phi_t}}\rho^\nabla(\omega) =  \Delta_{\omega_{t,\phi_t}}\left(\log \left(\frac{v_{\omega_{t,\phi_t}}}{v_{\omega}}\right)\right).$$
	Since the linearization of the volume form $v_{\omega_{t,\phi_t}}$ is an elliptic system in $\mathbb{G}_tJ_td\Delta^{g_t}\phi_t$, the volume form $v_{\omega_{t,\phi_t}}$ has the same regularity as the metric $\omega_{t,\phi_t}$ itself, as well as the left-hand side of this equation. The smooth regularity of the solution follows by the usual bootstrap argument via Schauder’s estimates for elliptic operators.
	
	We obtain then the following theorem:
	\begin{thm}\label{deformations}
		Let $(M,\omega,J,g)$ be a $4$-dimensional compact K\"ahler manifold such that the metric is a K\"ahler--Ricci soliton
		with respect to the vector field $\xi$. Let $T$ be a maximal torus in $Ham(M,\omega)$ such that the Lie algebra of $T$
		contains $\xi$. Let $J_t$ be any smooth family of $T$-invariant, $\omega$-compatible almost-complex structures such that $J_0=J$ and for small $t$, we have $h^-_{J_t}=b^+-1.$ Then, there exists a family of $T$-invariant, $\omega$-compatible almost-complex structures $\tilde{J}_t$ such that $(\omega,\tilde{J}_t)$ is GeAKRS for small enough $t$ and
		$\tilde{J}_0=J_0$. Moreover, $\tilde{J}_t$ is diffeomorphic to $J_t$ for each $t.$
	\end{thm}

\section{Toric Case}\label{toric-section}

	Let $(M,\omega)$ be a compact symplectic manifold of real dimension $2n$ equipped with an effective Hamiltonian action of an $n$-dimensional torus $T.$ We denote by $z:M\longrightarrow \Delta\subset \mathfrak{t}^\ast$ the moment map, where $\Delta$ is the Delzant polytope in $\mathfrak{t}^\ast$, the dual of  $\mathfrak{t}=Lie(T).$ A $T$-invariant almost-K\"ahler metric $g$ induced by an almost-complex structure $J$ compatible with $\omega$ is of involutive type if the $g$-orthogonal distribution to the $T$-orbits is involutive. 
	Given an almost-K\"ahler metric $g$ of involutive type, there exists local coordinates $\{t_1,\cdots,t_n\}$ such that $dt_1,\cdots,dt_n$ span the annihilator of the $g$-orthogonal distribution to the $T$-orbits and we can write
	$$\omega=\sum_{i=1}^ndz_i\wedge dt_i,$$
	$$g=\sum_{i,j=1}^n G_{ij}(z)dz_i\otimes dz_j+H_{ij}(z)dt_i\otimes dt_j,$$
	where $G=\left(G_{ij}(z)\right)$ and $H=\left(H_{ij}(z)\right)$ are symmetric positive-definite matrix-valued functions which are smooth in the interior of $\Delta$ and satisfy $H=G^{-1}.$ Necessary and sufficient conditions for $H$ to be induced by a globally defined almost-K\"ahler metric are given in~\cite{MR1895351,MR2154300,MR2144249}. The metric $g$ is K\"ahler if and only if $$\frac{\partial G_{ij}}{\partial z_k}=\frac{\partial G_{kj}}{\partial z_i}.$$ The first-Chern--Ricci form and the Chern scalar curvature are given by~\cite{MR1988506,MR2747965}
	$$\rho^\nabla=-\frac{1}{2}\sum_{i,k,l}H_{li,ik}dz_k\wedge dt_l,$$
	$$s^c=-\frac{1}{2}\sum_{i,j}H_{ij,ij},$$
	where $H_{ij,k}=\frac{\partial H_{ij}}{\partial z_k},$ etc. A Hamiltonian Killing potential $f_\xi$ can be expressed
	as $f_\xi=a_1z_1+a_2z_2+\cdots+a_nz_n+a_{n+1},$ where $a_i$ are real numbers. Then we can compute
	$$dJdf_\xi=\sum_{i,l,j=1}^na_iH_{il,k}dz_k\wedge dt_l.$$ 
	Moreover, the following hold:
	$$\Delta^gf_\xi=-\sum_{i,j}^na_iH_{ij,j}, \quad \text{and } \quad |df_\xi|^2_g=\sum_{i,j}^na_ia_jH_{ij}.$$
	Putting these equations together, we obtain a formula for the modified Chern scalar curvature in the toric case:
	$$ s^c_\xi = -\frac{1}{2}\sum_{i,j}H_{ij,ij} + 2\sum_{i,j}^na_iH_{ij,j} - 2 \sum_{i,j}^na_ia_jH_{ij} + 2\left(\sum_{i}^n a_i z_i + a_{n+1}\right) - n .$$
	It is easy to check then that
	\begin{equation}\label{toric-equation}
		e^{-2f_\xi}s^c_\xi=-\frac{1}{2}\sum_{i,j=1}^n\left(e^{-2f_\xi}H_{ij}\right)_{,ij}+e^{-2f_\xi}(2f_\xi-n) 
	\end{equation}
	We suppose that the metric $g$ is a GeAKRS. Since the metric $g$ is of involutive type, the matrix $H$ satisfies the boundary conditions in~\cite[Proposition 1]{MR2144249}, so we can apply~\cite[Lemma 2]{MR3941493} and Equation~\ref{toric-equation} becomes
	$$\int_{\partial\Delta}e^{-2f_\xi}d\mu+\int_\Delta (2f_\xi-n)e^{-2f_\xi}dv=0,$$
	where $\Delta$ is the polytope and $\partial\Delta$ its boundary, $dv=dz_1\wedge\cdots\wedge dz_n$ and $d\mu$ is defined by $u_j\wedge d\mu=-dv$ for any codimension one face with inward normal $u_j.$ 

	The Donaldson--Futaki invariant~\cite{MR1988506} for any function $\zeta=\zeta(z_1, \cdots, z_n)$ is
	$$\mathcal{F}_{\Delta,\xi}(\zeta)=\int_{\partial\Delta}\zeta e^{-2f_\xi}d\mu+\int_\Delta \zeta (2f_\xi-n)e^{-2f_\xi}dv.$$
	Using~\cite[Lemma 2]{MR3941493}, we can deduce that the existence of a GeAKRS implies that $\mathcal{F}_{\Delta,\xi}(\zeta)=0$
	for any affine function $\zeta=\zeta(z_1,\cdots,z_n)$.


	\begin{lem}\label{toric-deformations}
		Let $(M,\omega,J_0,g_0)$ be a $2n$-dimensional compact toric K\"ahler Fano manifold such that the metric is $T$-invariant and it is a K\"ahler--Ricci soliton with respect to the vector field $\xi$. Then, there exists an infinite-dimensional family of strictly almost-K\"ahler metrics $g_t$ (associated to almost-complex structures $J_t$ through $\omega$), defined for small $t$, such that $$\rho^\nabla_t-\omega=dJ_td f_\xi,$$ where $\rho^\nabla_t$ is the first-Chern--Ricci form of $(\omega,J_t,g_t).$ In particular $g_t$ is a GeAKRS.
	\end{lem}
	\begin{proof}
		The metric $g_0$ is of involutive type. We consider a deformation of the metric of the form $H_t=H+tD$
		and $G_t=\left(H_t\right)^{-1}$, where $D=(D_{ij}(z))$ is symmetric and has compact support in an open set $U$ of $M.$ Since $g_0$ is a K\"ahler--Ricci soliton, solving the equation $\rho^\nabla_t-\omega=dJ_td f_\xi$ gives the system 
		\begin{equation*}
		\frac{1}{2}\sum_{i=1}^nD_{il,ik}+\sum_{i=1}^na_iD_{il,k}=0,
		\end{equation*}
		for $k,l=1,\cdots,n$. Here $D_{il,ik}$ denotes $\frac{\partial^2D_{il}}{\partial z_i\partial z_k}$ etc. It is enough then to solve the system
		\begin{equation}\label{toric-equ}
		\frac{1}{2}\sum_{i=1}^nD_{il,i}+\sum_{i=1}^na_iD_{il}=0,
		\end{equation}
		for $l=1,\cdots,n$.
		We define $V_{il}:=D_{il}\,e^{2f_\xi}.$ Remark that if $V_{il}$ is of compact support in $U$ then $D_{il}$ is of compact support in $U$. Equation~(\ref{toric-equ}) is then equivalent to
		\begin{equation*}
		\sum_{i=1}^nV_{il,i}=0,
		\end{equation*}
		for $l=1,\cdots,n$.
		We remark that the function $V_{ll}$, for a fixed $l$ is determined by the functions $V_{il}$, with $i=1,\cdots,n$ and $i\neq l $. For $1\leq i<l\leq n$, we can then choose functions $V_{il}$ of the form $V_{il}=u_{il}(z_i)v_{il}(z_l)w_{il}(z_1,\cdots,\hat{z_i},\cdots,\hat{z_l},\cdots,z_n),$ (here $\hat{z_i}$ means $z_i$ is removed etc) such that the functions $u_{il}(z_i),v_{il}(z_l)$ and $w_{il}(z_1,\cdots,\hat{z_i},\cdots,\hat{z_l},\cdots,z_n)$ have compact support in $U$ with $\int_Uu_{il}(z_i)\left(\omega|_U\right)^n=\int_Uv_{il}(z_i)\left(\omega|_U\right)^n=0$ so that $u_{il}(z_i)$ and $v_{il}(z_l)$ admit primitives with compact support in $U$. For generic choices of such functions, the metric is not K\"ahler.
	\end{proof}
	We can then deduce the following:
	\begin{cor}\label{existence-toric-dim4}
Any $2n$-dimensional compact toric symplectic Fano manifold admits a GeAKRS which is strictly non-K\"ahler.
	\end{cor}
	\begin{proof}
		This is a consequence of Lemma~\ref{toric-deformations}, the fact that any toric symplectic manifold is K\"ahler~\cite{MR984900}, and the existence of K\"ahler--Ricci soliton on any toric K\"ahler Fano manifold~\cite{MR2084775}.
	\end{proof}

  \section{Holomorphic vector fields on GeAKRS manifolds}\label{sec-lie-algebra}
  
  A structure theorem for the automorphism group on K\"ahler manifolds admitting K\"ahler--Ricci solitons is given in~\cite{MR1768112}. In particular,
  it is known that Killing vector fields are Hamiltonian on such manifolds. We extend some of these results to symplectic Fano GeAKRS manifolds of real dimension $2n.$ We recall that on an almost-K\"ahler manifold $(M,\omega,J)$, a (real) vector field $X$ is called holomorphic if it preserves the almost-complex structure $J$ i.e. $\mathcal{L}_XJ=0$. This is equivalent to $\bar{\partial}^{TM}X-N(X,\cdot)=0,$ where $\bar{\partial}^{TM}$ is the Cauchy--Riemann operator defined on the tangent bundle and $N$ is the Nijenhuis tensor of $J$. We deduce that on a compact almost-K\"ahler manifold, the Lie algebra of holomorphic vector fields is finite-dimensional~\cite[Theorem 4.1]{MR1336823}. Moreover, since $\omega$ is $\Delta^g$-harmonic, any Killing vector field is holomorphic.
 

\begin{thm}\label{killing-algebra}
Let $(M,\omega,J,g)$ be a compact symplectic Fano GeAKRS manifold of real dimension $2n > 2$ such that $\rho^\nabla-\omega=-dd^cf_\xi.$ Let $X$ be holomorphic vector field. Then the Riemannian dual of the vector field $X$ is of the form$$\alpha=d^cu+dh,$$ where $u,h$ are functions normalized such that $\int_Mue^{-2f_\xi}\omega^n=\int_Mhe^{-2f_\xi}\omega^n=0$, and furthermore satisfy
 \begin{eqnarray}
 \frac{1}{2}\Delta^gu& =&u-g\left(\alpha,d^cf_\xi\right),\label{result-1}\\
 \frac{1}{2}\Delta^g h&=&h-g(J\alpha,d^cf_\xi),\label{result-2}\\
 D^g_{(dh)^\sharp}\omega&=&0,\label{result-3}
\end{eqnarray}
where $\Delta^{{g}}$ is the Riemannian Laplacian of the metric $g$ and $\sharp$ is the $g$-Riemannian dual. Furthermore, if $X$ is a Killing vector field, then $X$ is a Hamiltonian vector field such that the Killing potential $u$ satisfies 
 \begin{equation}\label{killing-potential-conf-eq}
\Delta^{\tilde{g}}u=2u,
 \end{equation}
 where $\Delta^{\tilde{g}}$ is the Riemannian Laplacian of the conformal metric $\tilde{g}=e^{\frac{-2f_\xi}{n-1}}g.$ In particular, there are no non-trivial Killing vector fields if $2$ is not an eigenvalue of $\Delta^{\tilde{g}}.$
 \end{thm}
 \begin{proof}
 The Riemannian dual $\alpha$ of $X$ admits the following Hodge decomposition with respect to the twisted Laplacian $\Delta^c=J\Delta^gJ^{-1}$~\cite[Corollary 1]{MR3331165}
 $$\alpha=\alpha_{H^c}+d^cu-J\mathbb{G}d^cv,$$
 for some functions $u,v$ (defined up to a constant), where $(\cdot)_{H^c}$ is the harmonic part with respect to $\Delta^c$ and $\mathbb{G}$ is the Green operator associated to $\Delta^g.$ Now, the vector field $X$ preserves $J$ so $\mathcal{L}_X\omega=dJ\alpha=-dJJ\mathbb{G}d^cv=d\mathbb{G}d^cv$ is $J$-invariant.
Moreover, since $X$ is a real holomorphic vector field, we see that ~\cite[Corollary 2.4]{MR2747965}
 \begin{equation}\label{bochner}
 \delta^g\left(D^g\alpha\right)^{J,+}-\delta^g\left(D^g\alpha\right)^{J,-}=\rho^\nabla\left(X,J\cdot\right),
 \end{equation}
 where $D^g$ is the Levi-Civita connection of $g$ and $(\cdot)^{J,+}$ (resp. $(\cdot)^{J,-}$) denotes the $J$-invariant (resp. the $J$-anti-invariant) part of the $2$-tensor. Furthermore, since $X$ is holomorphic we know that $\left(D^g\alpha\right)^{J,-}=-\frac{1}{2}D^g_{JX}J$ is anti-symmetric~\cite{MR3331165} so the $g$-symmetric part $\left(D^g\alpha\right)^{sym}$ of $D^g\alpha$ is $J$-invariant and it is given by $\left(D^g\alpha\right)^{sym}=-\frac{1}{2}(dJ\alpha)_{J\cdot,\cdot}=-\frac{1}{2}(d\mathbb{G}d^cv)_{J\cdot,\cdot}.$

  Denote by $(\cdot)^{sym}$ the $g$-symmetric part. Then,  
 Equation~(\ref{bochner}) reduces to
 \begin{eqnarray*}
  \delta^g\left(D^g\alpha\right)^{J,+}-\delta^g\left(D^g\alpha\right)^{J,-}&=& \delta^g\left(D^g\alpha\right)^{sym}+\frac{1}{2}\delta^g\left(\left(d\alpha\right)^{J,+}-\left(d\alpha\right)^{J,-}\right),\\
  &=& \delta^g\left(D^g\alpha\right)^{sym}+\frac{1}{2}\delta^g\left(J\left(d\alpha\right)^{J,+}+J\left(d\alpha\right)^{J,-}\right),\\
  &=& \delta^g\left(D^g\alpha\right)^{sym}+\frac{1}{2}\delta^gJ\left(d\alpha\right),\\
  &=& \delta^g\left(D^g\alpha\right)^{sym}-\frac{1}{2}J\delta^cd\alpha,\\
  &=& \delta^g\left(D^g\alpha\right)^{sym}+\frac{1}{2}Jd\delta^c\alpha,\\
   &=& \delta^g\left(D^g\alpha\right)^{sym}+\frac{1}{2}Jd\left(\Delta^gu\right),\\
   &=& \delta^g\left(D^g\alpha\right)^{sym}+\frac{1}{2}d^c\Delta^gu.
   \end{eqnarray*}
   On the other hand, the right-hand side of Equation~(\ref{bochner}) is given by
   \begin{eqnarray*}
   \rho^\nabla\left(X,J\cdot\right)&=&\omega(X,J)-dd^cf_\xi\left(X,J\cdot\right),\\
   &=&\alpha+J\mathcal{L}_Xd^cf_\xi-d^c\left(g(\alpha,d^cf_\xi)\right),\\
   &=&\alpha-d\mathcal{L}_Xf_\xi-d^c\left(g(\alpha,d^cf_\xi)\right),\\
   &=&\alpha-d\left(g(J\alpha,d^cf_\xi)\right)-d^c\left(g(\alpha,d^cf_\xi)\right).
   \end{eqnarray*}
   where $\mathcal{L}$ is the Lie derivative. Hence, using the fact that $\left(D^g\alpha\right)^{sym}=-\frac{1}{2}(d\mathbb{G}d^cv)_{J\cdot,\cdot}$ is $J$-invariant, Equation~(\ref{bochner}) becomes
   \begin{eqnarray}
   \delta^g\left(D^g\alpha\right)^{sym}+\frac{1}{2}d^c\Delta^gu&=&\alpha-d\left(g(J\alpha,d^cf_\xi)\right)-d^c\left(g(\alpha,d^cf_\xi)\right),\nonumber\\
  -\frac{1}{2} \delta^g\left(\left(d\mathbb{G}d^cv\right)_{J\cdot,\cdot}\right)+\frac{1}{2}d^c\Delta^gu&=&\alpha-d\left(g(J\alpha,d^cf_\xi)\right)-d^c\left(g(\alpha,d^cf_\xi)\right),\nonumber\\
   -\frac{1}{2}J \delta^g\left(d\mathbb{G}d^cv\right)+\frac{1}{2}d^c\Delta^gu&=&\alpha-d\left(g(J\alpha,d^cf_\xi)\right)-d^c\left(g(\alpha,d^cf_\xi)\right),\nonumber\\
    -\frac{1}{2}J\Delta^g \mathbb{G}\left(d^cv\right)+\frac{1}{2}d^c\Delta^gu&=&\alpha-d\left(g(J\alpha,d^cf_\xi)\right)-d^c\left(g(\alpha,d^cf_\xi)\right),\nonumber\\
    -\frac{1}{2}J d^cv+\frac{1}{2}d^c\Delta^gu&=&\alpha-d\left(g(J\alpha,d^cf_\xi)\right)-d^c\left(g(\alpha,d^cf_\xi)\right),\nonumber\\
     \frac{1}{2} dv+\frac{1}{2}d^c\Delta^gu&=&\alpha-d\left(g(J\alpha,d^cf_\xi)\right)-d^c\left(g(\alpha,d^cf_\xi)\right).\label{bochner-equation}
  \end{eqnarray}
  
 Since, we have $-J\mathbb{G}d^cv=J\delta^g\mathbb{G}(v\omega)$ and $\delta^g\left(\alpha_{H^c}\right)=0$, we deduce that
  $\alpha_{H^c}=0$ and so $$\alpha=d^cu-J\mathbb{G}d^cv.$$ Moreover, we apply $d^c$ and $d$ to obtain
  \begin{eqnarray*}
   \frac{1}{2} d^cdv&=&-d\mathbb{G}d^cv-d^cd\left(g(J\alpha,d^cf_\xi)\right),\\
    \frac{1}{2} dd^c\Delta^gu&=&dd^cu-dJ\mathbb{G}d^cv-dd^c\left(g(\alpha,d^cf_\xi)\right).
   \end{eqnarray*}
  We take the trace of the two equations with respect to $\omega$. Because $d\mathbb{G}d^cv$ is $J$-invariant, we have $\Lambda d\mathbb{G}d^cv=-v+c$
  for some constant $c$ (see for instance~\cite{MR3331165}). Moreover, $\Lambda d J\mathbb{G}d^cv=-\delta^c J\mathbb{G}d^cv=-\delta^c J\delta^g\mathbb{G}(v\omega)=0$. We deduce that
  
  \begin{eqnarray}
   \frac{1}{2}\Delta^gu&=&u-g\left(\alpha,d^cf_\xi\right)+c_1\label{GeAKRS eq_u},\\
    \frac{1}{2}\Delta^g v&=&v+c_2-\Delta^g\left(g(J\alpha,d^cf_\xi)\right)\label{GeAKRS eq_v},
  \end{eqnarray}
  for some constants $c_1,c_2.$ We normalize $u$ so that $\int_M u\,e^{-2f_\xi}\,\omega^n=0.$ We multiply Equation~(\ref{GeAKRS eq_u}) by $e^{-2f_\xi}$ and we integrate. We get

 \begin{eqnarray*}
 \frac{1}{2}\int_M \Delta^gue^{-2f_\xi}\omega^n+\int_Mg\left(\alpha,d^cf_\xi\right)e^{-2f_\xi}\omega^n&=&c_1\int_Me^{-2f_\xi}\omega^n,\\
  \frac{1}{2}\int_M \Delta^gue^{-2f_\xi}\omega^n+\int_Mg\left(d^cu,d^cf_\xi\right)e^{-2f_\xi}\omega^n&=&c_1\int_Me^{-2f_\xi}\omega^n,\\
  0&=&c_1\int_Me^{-2f_\xi}\omega^n.
  \end{eqnarray*}
  We conclude that $c_1=0$ and we obtain~(\ref{result-1}). On the other hand, we suppose that $\int_Mv\,\omega^2=0$ and we introduce
  the function $h$ satisfying $v=\Delta^gh.$ We normalize $h$ so that $\int_M h\,e^{-2f_\xi}\,\omega^2=0.$ Equation~(\ref{GeAKRS eq_v}) becomes 
$$ \frac{1}{2}\Delta^g h=h-g(J\alpha,d^cf_\xi)+c_3,$$
 for some constant $c_3.$ We multiply by $e^{-2f_\xi}$ and integrate. We get
  \begin{eqnarray*}
 \frac{1}{2}\int_M \Delta^ghe^{-2f_\xi}\omega^n+\int_Mg\left(J\alpha,d^cf_\xi\right)e^{-2f_\xi}\omega^n&=&c_3\int_Me^{-2f_\xi}\omega^n,\\
 \frac{1}{2}\int_M \Delta^ghe^{-2f_\xi}\omega^n-\frac{1}{2}\int_Mg\left(\alpha,d(e^{-2f_\xi})\right)\omega^n&=&c_3\int_Me^{-2f_\xi}\omega^n,\\
   \frac{1}{2}\int_M \Delta^ghe^{-2f_\xi}\omega^n+\frac{1}{2}\int_Mg\left(J\mathbb{G}d^cv,d(e^{-2f_\xi})\right)\omega^n&=&c_3\int_Me^{-2f_\xi}\omega^n,\\
      \frac{1}{2}\int_M \Delta^ghe^{-2f_\xi}\omega^n+\frac{1}{2}\int_Mg\left(\delta^gJ\mathbb{G}d^cv,e^{-2f_\xi}\right)\omega^n&=&c_3\int_Me^{-2f_\xi}\omega^n,\\
       \frac{1}{2}\int_M \Delta^ghe^{-2f_\xi}\omega^n+\frac{1}{2}\int_Mg\left(\Lambda d^c J\mathbb{G}d^cv,e^{-2f_\xi}\right)\omega^n&=&c_3\int_Me^{-2f_\xi}\omega^n,\\
         \frac{1}{2}\int_M  \Delta^ghe^{-2f_\xi}\omega^n+\frac{1}{2}\int_Mg\left(\Lambda d \mathbb{G}d^cv,e^{-2f_\xi}\right)\omega^n&=&c_3\int_Me^{-2f_\xi}\omega^n,\\
          \frac{1}{2}\int_M  \Delta^ghe^{-2f_\xi}\omega^n-\frac{1}{2}\int_Mve^{-2f_\xi}\omega^n&=&c_3\int_Me^{-2f_\xi}\omega^n,\\
           \frac{1}{2}\int_M  \Delta^ghe^{-2f_\xi}\omega^n-\frac{1}{2}\int_M \Delta^ghe^{-2f_\xi}\omega^n&=&c_3\int_Me^{-2f_\xi}\omega^n,\\
  0&=&c_3\int_Me^{-2f_\xi}\omega^n.
  \end{eqnarray*}
  We deduce that $c_3=0$ and so $$ \frac{1}{2}\Delta^g h=h-g(J\alpha,d^cf_\xi),$$
  where $\Delta^gh=v$ and we get~(\ref{result-2}). From Equation~(\ref{bochner-equation}), we obtain that
  \begin{eqnarray*}
  -J\mathbb{G}d^cv&=&  \frac{1}{2} dv+d\left(g(J\alpha,d^cf_\xi)\right),\\
  &=&  \frac{1}{2} d\Delta^gh+d\left(g(J\alpha,d^cf_\xi)\right),\\
  &=&dh.
  \end{eqnarray*}
  We conclude then that $$\alpha=d^cu+dh.$$
  Furthermore, since $\left(D^g\alpha\right)^{J,-}=-\frac{1}{2}D^g_{JX}J$ is anti-symmetric, we have
  $\left(d\alpha\right)^{J,-}=-D^g_{JX}\omega$. On the other hand $d\alpha=dd^cu$ and so $\left(d\alpha\right)^{J,-}=\left(dd^cu\right)^{J,-}=D^g_{(du)^\sharp}\omega$. We deduce that
  $D^g_{\left(dh\right)^\sharp}\omega=0.$ Finally, if $X$ is a Killing vector field then $\mathcal{L}_X\omega=dd^ch=0$ and so $h\equiv 0.$ Equation~(\ref{result-1}) is then equivalent to $\Delta^{\tilde{g}}u=2u.$
 
 \end{proof}
 
 \begin{rem}\label{remark1}
We observe that in the K\"ahler case, a converse to Theorem~(\ref{killing-algebra}) holds (in any dimension). More explicitly, if we have a K\"ahler-Ricci soliton and functions $u,h$ satisfying Equations~(\ref{result-1}) and~(\ref{result-2}) (in the K\"ahler case, Equation~(\ref{result-3}) is automatically satisfied) then the vector field $X$ with Riemannian dual $\alpha=d^cu+dh$ is a holomorphic vector field (see~\cite{MR1768112}) and so one can use Equations~(\ref{result-1}) and~(\ref{result-2}) to deduce a structure theorem of the Lie algebra of holomorphic vector fields~~\cite{MR1768112}. However, in the strictly almost-K\"ahler case, the assumption that $$\int_M\left(D^gX\right)^{sym,J,-}\left(e_i,\left(D^g_{e_i}J\right)JX)\right)) \,e^{2f_\xi}\omega^n=0,$$
is needed to prove the converse, where $\left(D^gX\right)^{sym,J,-}$ is the $g$-symmetric $J$
-anti-invariant part of of $D^gX.$ The above integral clearly vanishes in the K\"ahler case.
\end{rem}

In the toric case, we can prove a converse to Theorem~\ref{killing-algebra} in the almost-K\"ahler case (see for instance~\cite{MR3833798} in the K\"ahler--Einstein case)
\begin{prop}
Suppose that $(M,\omega,J,g)$ is a compact toric almost-K\"ahler manifold of real dimension $2n > 2$ such that the metric $g$ is $T$-invariant and of involutive type. Let $z:M\longrightarrow \Delta\subset \mathfrak{t}^\ast$ be the moment map, where $\Delta$ is the Delzant polytope. Suppose the moment coordinate $z_i$ satisfies $$\Delta^{\tilde{g}}z_i=2z_i,$$
for each $i=1,\cdots,n$, where $\tilde{g}$ is the conformal metric $\tilde{g}=e^{-\frac{2f_\xi}{n-1}}g$ with $f_\xi=a_1z_1+a_2z_2+\cdots+a_nz_n+a_{n+1},$ for some real numbers $a_i$. Then $$\rho^\nabla-\omega=-dd^cf_\xi,$$
so the almost-K\"ahler metric $g$ is a GeAKRS.
 \end{prop}
\begin{proof}
We recall that $$\omega=\sum_{i=1}^ndz_i\wedge dt_i,$$ and
$$g=\sum_{i,j=1}^n G_{ij}(z)dz_i\otimes dz_j+H_{ij}(z)dt_i\otimes dt_j.$$
Then, we have for each $i$
$$\Delta^g z_i=-\sum_{j}^nH_{ij,j}$$
and 
$$g(dz_i,df_\xi)=\sum_{j=1}^na_jH_{ni}.$$
By hypothesis, we have $\Delta^{\tilde{g}}z_i=2z_i,$ thus
\begin{equation}\label{toric-eq}
-\sum_{j}^nH_{ij,j}+2\sum_{j=1}^na_jH_{ji}=2z_i.
\end{equation}
Using Equation~(\ref{toric-eq}) we get the following
\begin{eqnarray*}
\rho^\nabla&=&-\frac{1}{2}\sum_{j,l,k}H_{lj,jk}\,dz_k\wedge dt_l,\\
&=&-\frac{1}{2}\sum_{j,l,k}\frac{\partial}{\partial z_k}\left(H_{lj,j}\right)\,dz_k\wedge dt_l,\\
&=&-\frac{1}{2}\sum_{j,l,k}\frac{\partial}{\partial z_k}\left(-2z_l+2\sum_{j=1}^na_jH_{jl}\right)\,dz_k\wedge dt_l,\\
&=&\sum_kdz_k\wedge dt_k-\sum_{j=1}^na_jH_{jl,k}\,dz_k\wedge dt_l,\\
&=&\omega-dd^cf_\xi.
\end{eqnarray*}

\end{proof}

For a general GeAKRS, we can show the following using the proof of Theorem~\ref{killing-algebra}

\begin{prop}
Let $(M,\omega,J,g)$ be a compact connected symplectic Fano GeAKRS manifold of real dimension $2n$. Let $X$ be a Killing vector field. Then $X$ is a Hamiltonian Killing vector field such that the Killing potential $u$ satisfies
 \begin{eqnarray}
 \frac{1}{2}\Delta^gu&=&u-g\left(d^cu,a\right)\label{result-4},
 \end{eqnarray}
with $\int_Mue^{-2f_\xi}\omega^n=0$, where $a$ is a $1$-form satisfying $\rho^\nabla-\omega=da$ and $\mathcal{L}_Xa=0$ (here $\mathcal{L}$ denotes the Lie derivative).
 \end{prop}
 \begin{proof}
 Since $X$ is a Killing vector, the Riemannian dual $\alpha$ of $X$ is of the form
 $$\alpha=\alpha_{H^c}+d^cu.$$ Moreover, $\left(D^g\alpha\right)^{sym}=0$. Hence, as in the proof of Theorem~\ref{killing-algebra}, the left-hand side of Equation~(\ref{bochner}) reduces to
 $$\delta^g\left(D^g\alpha\right)^{J,+}-\delta^g\left(D^g\alpha\right)^{J,-}=\frac{1}{2}d^c\Delta^gu.$$
 On the other hand, $\rho^\nabla-\omega=da$ for some $1$-form $a$. Since $M$ is connected, we can choose $a$ such that $\mathcal{L}_Xa=0$. Hence, using Cartan formula, the right-hand side of Equation~(\ref{bochner}) is given by
   \begin{eqnarray*}
   \rho^\nabla\left(X,J\cdot\right)&=&\omega(X,J\cdot)+da\left(X,J\cdot\right),\\
   &=&\alpha+d^c\left(g(a,\alpha)\right),
   \end{eqnarray*}
Equation~(\ref{bochner}) is now given by
$$\frac{1}{2}d^c\Delta^gu=\alpha+d^c\left(g(a,\alpha)\right).$$
We deduce that $\alpha_{H^c}=0$ and $\frac{1}{2}\Delta^gu=u+g\left(a,\alpha\right)+c$, for some constant $c$. Multiplying by $e^{-2f_\xi}$ and integrating, we get 
\begin{eqnarray*}
 \frac{1}{2}\int_M \Delta^gue^{-2f_\xi}\omega^n-\int_Mg\left(d^cu,a\right)e^{-2f_\xi}\omega^n&=&c\int_Me^{-2f_\xi}\omega^n,\\
  \frac{1}{2}\int_M \Delta^gue^{-2f_\xi}\omega^n+\int_Mg\left(du,Ja\right)e^{-2f_\xi}\omega^n&=&c\int_Me^{-2f_\xi}\omega^n,\\
  -\int_M u\left(\Delta^gf_\xi+2|df_\xi|^2_g\right)e^{-2f_\xi}\omega^n+\int_Mu\left(\delta^gJa+2g(Ja,df_\xi)\right)e^{-2f_\xi}\omega^n&=&c\int_Me^{-2f_\xi}\omega^n,\\
  0&=&c\int_Me^{-2f_\xi}\omega^n,
  \end{eqnarray*}
where we use in the last line that the metric is a GeAKRS. Therefore $c=0.$
 \end{proof}
 
 We call an almost-K\"ahler metric $(\omega,J,g)$ that satisfies the condition
 $$\rho^\nabla=\lambda\omega,$$ where $\lambda=0,1,-1$ a {\it{first-Chern--Einstein metric}} (such a metric is also called Hermite--Einstein or special in the literature; for more about these metrics in the non-K\"ahler setting, we refer the reader for instance to~\cite{MR3331165,MR2795448,MR2988734,MR4039808,MR4140765,MR4437352,MR4125707,MR4515775}). The proof of
 Theorem~\ref{killing-algebra} shows the following when the almost-K\"ahler metric is a first-Chern--Einstein metric. This can be compared to Matsushima's theorem~\cite{MR0094478} in the K\"ahler case
 \begin{cor}\label{AK-Matsushima}
 Let $(M,\omega,J,g)$ be a compact almost-K\"ahler manifold such that the metric $g$ is a first-Chern--Einstein metric. Then, we have the following
 \begin{enumerate}
 \item If $\rho^\nabla=0$, then any holomorphic vector field is a Killing vector field. Moreover,
 the Lie algebra of holomorphic vector fields is abelian and the Riemannian dual of any holomorphic vector field is harmonic with respect to the twisted Laplacian
 $\Delta^c=J\Delta^gJ^{-1}.$\\
 \item If $\rho^\nabla=-\omega$, then there are no non-trivial holomorphic vector fields on $M.$\\
 \item If $\rho^\nabla=\omega$, then the Riemannian dual $\alpha$ of a holomorphic vector field $X$ is of the form
 $$\alpha=d^cu+dh,$$
 such that 
 \begin{eqnarray*}
 \Delta^gu=2u,\\ 
 \Delta^gh=2h,\\
  D^g_{(dh)^\sharp}\omega=0.
 \end{eqnarray*}
 In particular, there are no non-trivial holomorphic vector fields if $2$ is not an eigenvalue of $\Delta^g.$ Moreover, if $X$ is a Killing vector field, then $X$ is a Hamiltonian vector field.
 \end{enumerate}
 \end{cor}
 \begin{proof}
Let $X$ be a holomorphic vector field. The Riemannian dual $\alpha$ of $X$ can be decomposed as $$\alpha=\alpha_{H^c}+d^cu-J\mathbb{G}d^cv.$$ From Equation~(\ref{bochner-equation}) in the proof of Theorem~\ref{killing-algebra}, we have  
\begin{equation}\label{einstein-equation}
 \frac{1}{2} dv+\frac{1}{2}d^c\Delta^gu=\rho^\nabla\left(X,J\cdot\right).
\end{equation}

If $\rho^\nabla=0$, then it follows from Equation~(\ref{einstein-equation}) that $v=u\equiv 0$, so $\alpha=\alpha_{H^c}$. It follows that
$\mathcal{L}_X\omega=dJ\alpha=0$ and so $X$ is a Killing vector field. Therefore $X$ preserves any harmonic $1$-form with respect to the twisted Laplacian $\Delta^c$. We deduce that the Lie algebra of holomorphic vector fields is abelian.

If $\rho^\nabla=-\omega$ then it follows from Equation~(\ref{einstein-equation}) that $\frac{1}{2} dv+\frac{1}{2}d^c\Delta^gu=-\alpha.$ Thus, $\alpha_{H^c}=0$ and  
\begin{eqnarray}
\frac{1}{2}d^c\Delta^gu&=&-d^cu,\label{negative-1}\\
\frac{1}{2} dv&=&J\mathbb{G}d^cv\label{negative-2}.
\end{eqnarray}
 Suppose that $u,v$ are normalized so that $\int_Mu\,\omega^n=\int_Mv\,\omega^n=0.$ It follows from Equation~(\ref{negative-1}) that $\frac{1}{2}\Delta^gu=-u$ and hence $u \equiv 0$ since $M$ is compact. Applying $dJ$ to Equation~(\ref{negative-2}) gives $\frac{1}{2} dd^cv=-d\mathbb{G}d^cv$.
 Taking the trace and since $d\mathbb{G}d^cv$ is $J$-invariant, we obtain $-\frac{1}{2}\Delta^gv=v$ and so $v\equiv 0.$ We conclude then that there are no non-trivial holomorphic vector fields on $M.$

If $\rho^\nabla=\omega$, then we apply Theorem~\ref{killing-algebra} with $f_\xi\equiv 0.$ The Riemannian dual $\alpha$ of $X$ can be expressed as $\alpha=d^cu+dh$. Moreover, Equations~(\ref{result-1}) and~(\ref{result-2}) reduce to $\Delta^gu=2u,$ and $\Delta^gh=2h.$ 
\end{proof}

%

 \begin{rem}\label{remark2}
 If $\rho^\nabla=\omega$ and $u$ is a function such that $\Delta^gu=2u$, then it follows from Lemma~\ref{lichnerowicz-AK} that
$\delta^g\left(Dd^cu\right)^{sym}_{J\cdot,\cdot}=0$. Howeover, this does not imply that $J\delta^g(Dd^cu)^{sym}$ vanishes (unless the metric is K\"ahler) so we can not conclude that $u$ is a Killing potential.
 \end{rem}

\appendix

\section{Proof of Lemma~\ref{lichnerowicz-AK}}

For any $1$-form $\alpha$, we have $(D^g\alpha)^{J,-}=\frac{1}{2}JD^g_{\alpha^\sharp} J-\frac{1}{2}J\mathfrak{L}_{\alpha^\sharp} J$, where $(\cdot)^{J,-}$ is the $J$-anti-invariant part. Moreover, $-J\mathfrak{L}_{\alpha^\sharp} J=2(D^g\alpha)^{sym}+(dJ\alpha)(J\cdot,\cdot).$ When $\alpha=d^cf$, we deduce that 
\begin{equation}\label{eq2}
(D^gd^cf)^{sym}=(D^gd^cf)^{J,-}-\frac{1}{2}D^g_{(df)^\sharp}\omega.
\end{equation}

On the other hand, from~\cite[Lemma 2.3]{MR2747965} and $\delta^g(D^gd^cf)=\frac{1}{2}\delta^gdd^cf+\delta^g(D^gd^cf)^{sym}$, we obtain 
\begin{equation*}
2\delta^g(D^gd^cf)^{J,-}=\frac{1}{2}\delta^gdd^cf+\delta^g(D^gd^cf)^{sym}-\rho^\ast\left(\left(d^cf\right)^\sharp,J\cdot\right)+\sum_{i=1}^{2n}(D^g_{Je_i}d^cf)\left((D^g_{e_i}J)(\cdot) \right),
\end{equation*}
where $\rho^\ast$ is the $\ast$-Ricci form defined as the image of the symplectic form $\omega$ under the Riemannian curvature operator (with respect to the Levi-Civita connection). Combining with Equation~(\ref{eq2}) and applying $J$, we obtain
\begin{align}\label{eq1}
J\delta^g(D^gd^cf)^{sym}=&\, \frac{1}{4}J\delta^gdd^cf+\frac{1}{2}J\delta^g(D^gd^cf)^{sym}-\frac{1}{2}\rho^\ast\left(\left(d^cf\right)^\sharp,\cdot\right)\\ 
& -\frac{1}{2}\sum_{i=1}^{2n}(D^g_{e_i}d^cf)\big{(}(D^g_{e_i}J)(\cdot) \big{)}-\frac{1}{2}J\delta^gD^g_{(df)^{\sharp}}\omega\nonumber.
\end{align}

On an almost-K\"ahler manifold $dd^cf+d^cdf=2D^g_{\left(df\right)^\sharp}\omega$ for any function $f$, so 
\begin{equation*}
J\delta^gdd^cf+d\Delta^gf=2J\delta^gD^g_{\left(df\right)^\sharp}\omega.
\end{equation*}
Then Equation~(\ref{eq1}) becomes
\begin{equation}\label{eq3}
J\delta^g(D^gd^cf)^{sym}=-\frac{1}{2}d\Delta^gf-\rho^\ast\left(\left(d^cf\right)^\sharp,\cdot\right)-\sum_{i=1}^{2n}(D^g_{e_i}d^cf)\left((D^g_{e_i}J)(\cdot) \right).
\end{equation}
Now recall the following relation between the $\ast$-Ricci form and the first-Chern--Ricci form $\rho^\nabla$ (see~\cite{MR1969266} for instance):
$$\rho^\nabla(X,Y)=\rho^\ast(X,Y)-\frac{1}{4}\tr\left(JD^g_XJ\circ D^g_YJ\right).$$
Since $(D^g_{e_i}d^cf)\left((D^g_{e_i}J)(\cdot) \right)=(D^gd^cf)^{J,-}\left(e_i,(D^g_{e_i}J)(\cdot) \right)$, Equation~(\ref{eq3}) becomes
$$J\delta^g(D^gd^cf)^{sym}=-\frac{1}{2}d\Delta^gf-\rho^\nabla\left(\left(d^cf\right)^\sharp,\cdot\right)-\frac{1}{4}\tr\left(D^g_{\left(df\right)^\sharp}J\circ D^g_{\cdot}J\right)-\sum_{i=1}^{2n}(D^gd^cf)^{J,-}\left(e_i,(D^g_{e_i}J)(\cdot) \right).$$
Using Equation~(\ref{eq2}), we obtain
\begin{eqnarray*}
J\delta^g(D^gd^cf)^{sym}&=&-\frac{1}{2}d\Delta^gf-\rho^\nabla\left(\left(d^cf\right)^\sharp,\cdot\right)-\frac{1}{4}\tr\left(D^g_{\left(df\right)^\sharp}J\circ D^g_{\cdot}J\right)-\sum_{i=1}^{2n}(D^gd^cf)^{sym}\left(e_i,(D^g_{e_i}J)(\cdot) \right)\\
&-&\frac{1}{2}\sum_{i=1}^{2n}\left(D^g_{(df)^\sharp}\omega\right)\left(e_i,(D^g_{e_i}J)(\cdot) \right),\\
&=&-\frac{1}{2}d\Delta^gf-\rho^\nabla\left(\left(d^cf\right)^\sharp,\cdot\right)-\frac{1}{4}\tr\left(D^g_{\left(df\right)^\sharp}J\circ D^g_{\cdot}J\right)-\sum_{i=1}^{2n}(D^gd^cf)^{sym}\left(e_i,(D^g_{e_i}J)(\cdot) \right)\\
&-&\frac{1}{2}\sum_{i,k=1}^{2n}\left(D^g_{(df)^\sharp}\omega\right)(e_i,e_k)\otimes \left(D^g_{e_i}\omega\right)(\cdot,e_k),\\
&=&-\frac{1}{2}d\Delta^gf-\rho^\nabla\left(\left(d^cf\right)^\sharp,\cdot\right)-\frac{1}{4}\tr\left(D^g_{\left(df\right)^\sharp}J\circ D^g_{\cdot}J\right)-\sum_{i=1}^{2n}(D^gd^cf)^{sym}\left(e_i,(D^g_{e_i}J)(\cdot) \right)\\
&-&\frac{1}{4}\sum_{i,k=1}^{2n}\left(D^g_{(df)^\sharp}\omega\right)(e_i,e_k)\otimes \left[\left(D^g_{e_i}\omega\right)(\cdot,e_k)-\left(D^g_{e_k}\omega\right)(\cdot,e_i)\right],\\
&=&-\frac{1}{2}d\Delta^gf-\rho^\nabla\left(\left(d^cf\right)^\sharp,\cdot\right)-\frac{1}{4}\tr\left(D^g_{\left(df\right)^\sharp}J\circ D^g_{\cdot}J\right)-\sum_{i=1}^{2n}(D^gd^cf)^{sym}\left(e_i,(D^g_{e_i}J)(\cdot) \right)\\
&-&\frac{1}{4}\sum_{i,k=1}^{2n}\left(D^g_{(df)^\sharp}\omega\right)(e_i,e_k)\otimes\left(D^g_{(df)^\sharp}\omega\right)(e_i,e_k),\\
&=&-\frac{1}{2}d\Delta^gf-\rho^\nabla\left(\left(d^cf\right)^\sharp,\cdot\right)-\frac{1}{4}\tr\left(D^g_{\left(df\right)^\sharp}J\circ D^g_{\cdot}J\right)-\sum_{i=1}^{2n}(D^gd^cf)^{sym}\left(e_i,(D^g_{e_i}J)(\cdot) \right)\\
&+&\frac{1}{4}\tr\left(D^g_{\left(df\right)^\sharp}J\circ D^g_{\cdot}J\right),\\
&=&-\frac{1}{2}d\Delta^gf-\rho^\nabla\left(\left(d^cf\right)^\sharp,\cdot\right)-\sum_{i=1}^{2n}(D^gd^cf)^{sym}\left(e_i,(D^g_{e_i}J)(\cdot) \right).
\end{eqnarray*}
Because $(D^gd^cf)^{sym}$ is $J$-anti-invariant, we obtain
\begin{equation*}
J\delta^g(D^gd^cf)^{sym}+\sum_{i=1}^{2n}(D^gd^cf)^{sym}\left(e_i,(D^g_{e_i}J)(\cdot) \right)=\delta^g\left(\left(D^gd^cf\right)^{sym}_{J\cdot,\cdot}\right).
\end{equation*}
Lemma~\ref{lichnerowicz-AK} follows.
  \bibliographystyle{abbrv}

\bibliography{solitons}

\begin{thebibliography}{10}

\bibitem{MR1895351}
M.~Abreu.
\newblock K\"{a}hler metrics on toric orbifolds.
\newblock {\em J. Differential Geom.}, 58(1):151--187, 2001.

\bibitem{MR4140765}
D.~V. Alekseevsky and F.~Podest\`a.
\newblock Homogeneous almost-{K}\"{a}hler manifolds and the {C}hern-{E}instein
  equation.
\newblock {\em Math. Z.}, 296(1-2):831--846, 2020.

\bibitem{MR4125707}
D.~Angella, S.~Calamai, and C.~Spotti.
\newblock Remarks on {C}hern-{E}instein {H}ermitian metrics.
\newblock {\em Math. Z.}, 295(3-4):1707--1722, 2020.

\bibitem{MR2144249}
V.~Apostolov, D.~M.~J. Calderbank, P.~Gauduchon, and C.~W.
  T{\o}nnesen-Friedman.
\newblock Hamiltonian 2-forms in {K}\"{a}hler geometry. {II}. {G}lobal
  classification.
\newblock {\em J. Differential Geom.}, 68(2):277--345, 2004.

\bibitem{MR2807093}
V.~Apostolov, D.~M.~J. Calderbank, P.~Gauduchon, and C.~W.
  T{\o}nnesen-Friedman.
\newblock Extremal {K}\"{a}hler metrics on projective bundles over a curve.
\newblock {\em Adv. Math.}, 227(6):2385--2424, 2011.

\bibitem{MR1969266}
V.~Apostolov and T.~Dr\u{a}ghici.
\newblock The curvature and the integrability of almost-{K}\"{a}hler manifolds:
  a survey.
\newblock In {\em Symplectic and contact topology: interactions and
  perspectives ({T}oronto, {ON}/{M}ontreal, {QC}, 2001)}, volume~35 of {\em
  Fields Inst. Commun.}, pages 25--53. Amer. Math. Soc., Providence, RI, 2003.

\bibitem{MR3941493}
V.~Apostolov and G.~Maschler.
\newblock Conformally {K}\"{a}hler, {E}instein-{M}axwell geometry.
\newblock {\em J. Eur. Math. Soc. (JEMS)}, 21(5):1319--1360, 2019.

\bibitem{MR4515775}
M.~Cahen, S.~Gutt, M.~Hayyani, and M.~Raouyane.
\newblock Some pseudo-{K}\"{a}hler {E}instein 4-symmetric spaces with a
  ``twin'' special almost complex structure.
\newblock {\em Differential Geom. Appl.}, 86:Paper No. 101958, 23, 2023.

\bibitem{MR0645743}
E.~Calabi.
\newblock Extremal {K}\"{a}hler metrics.
\newblock In {\em Seminar on {D}ifferential {G}eometry}, volume No. 102 of {\em
  Ann. of Math. Stud.}, pages 259--290. Princeton Univ. Press, Princeton, NJ,
  1982.

\bibitem{MR1417944}
H.-D. Cao.
\newblock Existence of gradient {K}\"{a}hler-{R}icci solitons.
\newblock In {\em Elliptic and parabolic methods in geometry ({M}inneapolis,
  {MN}, 1994)}, pages 1--16. A K Peters, Wellesley, MA, 1996.

\bibitem{MR4039808}
A.~Della~Vedova.
\newblock Special homogeneous almost complex structures on symplectic
  manifolds.
\newblock {\em J. Symplectic Geom.}, 17(5):1251--1295, 2019.

\bibitem{MR984900}
T.~Delzant.
\newblock Hamiltoniens p\'{e}riodiques et images convexes de l'application
  moment.
\newblock {\em Bull. Soc. Math. France}, 116(3):315--339, 1988.

\bibitem{MR2795448}
A.~J. Di~Scala and L.~Vezzoni.
\newblock Chern-flat and {R}icci-flat invariant almost {H}ermitian structures.
\newblock {\em Ann. Global Anal. Geom.}, 40(1):21--45, 2011.

\bibitem{MR1622931}
S.~K. Donaldson.
\newblock Remarks on gauge theory, complex geometry and {$4$}-manifold
  topology.
\newblock In {\em Fields {M}edallists' lectures}, volume~5 of {\em World Sci.
  Ser. 20th Century Math.}, pages 384--403. World Sci. Publ., River Edge, NJ,
  1997.

\bibitem{MR1988506}
S.~K. Donaldson.
\newblock Scalar curvature and stability of toric varieties.
\newblock {\em J. Differential Geom.}, 62(2):289--349, 2002.

\bibitem{MR2154300}
S.~K. Donaldson.
\newblock Interior estimates for solutions of {A}breu's equation.
\newblock {\em Collect. Math.}, 56(2):103--142, 2005.

\bibitem{MR2483362}
S.~K. Donaldson.
\newblock K\"{a}hler geometry on toric manifolds, and some other manifolds with
  large symmetry.
\newblock In {\em Handbook of geometric analysis. {N}o. 1}, volume~7 of {\em
  Adv. Lect. Math. (ALM)}, pages 29--75. Int. Press, Somerville, MA, 2008.

\bibitem{MR1073369}
A.~Fujiki.
\newblock The moduli spaces and {K}\"{a}hler metrics of polarized algebraic
  varieties.
\newblock {\em S\={u}gaku}, 42(3):231--243, 1990.

\bibitem{MR1053910}
A.~Fujiki and G.~Schumacher.
\newblock The moduli space of extremal compact {K}\"{a}hler manifolds and
  generalized {W}eil-{P}etersson metrics.
\newblock {\em Publ. Res. Inst. Math. Sci.}, 26(1):101--183, 1990.

\bibitem{MR1456265}
P.~Gauduchon.
\newblock Hermitian connections and {D}irac operators.
\newblock {\em Boll. Un. Mat. Ital. B (7)}, 11(2, suppl.):257--288, 1997.

\bibitem{Gauduchon-book}
P.~Gauduchon.
\newblock Calabi's extremal {K}\"{a}hler metrics: An elementary introduction.
\newblock {\em Preprint}, 2010.

\bibitem{MR2366370}
D.~Guan.
\newblock Extremal solitons and exponential {$C^\infty$} convergence of the
  modified {C}alabi flow on certain {$\Bbb{C}{\rm P}^1$} bundles.
\newblock {\em Pacific J. Math.}, 233(1):91--124, 2007.

\bibitem{MR1327154}
D.~Z.-D. Guan.
\newblock Quasi-{E}instein metrics.
\newblock {\em Internat. J. Math.}, 6(3):371--379, 1995.

\bibitem{MR4620280}
J.~He and K.~Zheng.
\newblock Hermitian {C}alabi functional in complexified orbits.
\newblock {\em Internat. J. Math.}, 34(8):Paper No. 2350047, 44, 2023.

\bibitem{MR4017922}
E.~Inoue.
\newblock The moduli space of {F}ano manifolds with {K}\"{a}hler-{R}icci
  solitons.
\newblock {\em Adv. Math.}, 357:106841, 65, 2019.

\bibitem{MR1336823}
S.~Kobayashi.
\newblock {\em Transformation groups in differential geometry}.
\newblock Classics in Mathematics. Springer-Verlag, Berlin, 1995.
\newblock Reprint of the 1972 edition.

\bibitem{MR1145263}
N.~Koiso.
\newblock On rotationally symmetric {H}amilton's equation for
  {K}\"{a}hler-{E}instein metrics.
\newblock In {\em Recent topics in differential and analytic geometry}, volume
  18-{\rm I} of {\em Adv. Stud. Pure Math.}, pages 327--337. Academic Press,
  Boston, MA, 1990.

\bibitem{MR3897025}
A.~Lahdili.
\newblock Automorphisms and deformations of conformally {K}\"{a}hler,
  {E}instein-{M}axwell metrics.
\newblock {\em J. Geom. Anal.}, 29(1):542--568, 2019.

\bibitem{MR1274118}
C.~LeBrun and S.~R. Simanca.
\newblock Extremal {K}\"{a}hler metrics and complex deformation theory.
\newblock {\em Geom. Funct. Anal.}, 4(3):298--336, 1994.

\bibitem{MR3833798}
E.~Legendre and R.~Sena-Dias.
\newblock Toric aspects of the first eigenvalue.
\newblock {\em J. Geom. Anal.}, 28(3):2395--2421, 2018.

\bibitem{MR3078263}
E.~Legendre and C.~W. T{\o}nnesen-Friedman.
\newblock Toric generalized {K}\"{a}hler-{R}icci solitons with {H}amiltonian
  2-form.
\newblock {\em Math. Z.}, 274(3-4):1177--1209, 2013.

\bibitem{MR2747965}
M.~Lejmi.
\newblock Extremal almost-{K}\"{a}hler metrics.
\newblock {\em Internat. J. Math.}, 21(12):1639--1662, 2010.

\bibitem{MR2661166}
M.~Lejmi.
\newblock Stability under deformations of extremal almost-{K}\"{a}hler metrics
  in dimension 4.
\newblock {\em Math. Res. Lett.}, 17(4):601--612, 2010.

\bibitem{MR3331165}
M.~Lejmi.
\newblock Stability under deformations of {H}ermite-{E}instein almost
  {K}\"{a}hler metrics.
\newblock {\em Ann. Inst. Fourier (Grenoble)}, 64(6):2251--2263, 2014.

\bibitem{MR4184828}
M.~Lejmi and M.~Upmeier.
\newblock Integrability theorems and conformally constant {C}hern scalar
  curvature metrics in almost {H}ermitian geometry.
\newblock {\em Comm. Anal. Geom.}, 28(7):1603--1645, 2020.

\bibitem{MR3663320}
H.~Li.
\newblock Complex deformation of critical {K}\"{a}hler metrics.
\newblock {\em J. Math. Study}, 50(2):144--164, 2017.

\bibitem{MR66733}
P.~Libermann.
\newblock Sur les connexions hermitiennes.
\newblock {\em C. R. Acad. Sci. Paris}, 239:1579--1581, 1954.

\bibitem{MR0124009}
A.~Lichnerowicz.
\newblock {\em G\'{e}om\'{e}trie des groupes de transformations}.
\newblock Travaux et Recherches Math\'{e}matiques, III. Dunod, Paris, 1958.

\bibitem{MR0165458}
A.~Lichnerowicz.
\newblock {\em Th\'{e}orie globale des connexions et des groupes d'holonomie}.
\newblock Consiglio Nazionale delle Ricerche Monografie Matematiche, Vol. 2.
  Edizioni Cremonese, Rome; Dunod, Paris, 1962.

\bibitem{MR2805600}
G.~Maschler and C.~W. T{\o}nnesen-Friedman.
\newblock Generalizations of {K}\"{a}hler-{R}icci solitons on projective
  bundles.
\newblock {\em Math. Scand.}, 108(2):161--176, 2011.

\bibitem{MR0094478}
Y.~Matsushima.
\newblock Sur la structure du groupe d'hom\'{e}omorphismes analytiques d'une
  certaine vari\'{e}t\'{e} k\"{a}hl\'{e}rienne.
\newblock {\em Nagoya Math. J.}, 11:145--150, 1957.

\bibitem{MR1637093}
S.~A. Merkulov.
\newblock Formality of canonical symplectic complexes and {F}robenius
  manifolds.
\newblock {\em Internat. Math. Res. Notices}, (14):727--733, 1998.

\bibitem{MR2831983}
Y.~Nakagawa.
\newblock On generalized {K}\"{a}hler-{R}icci solitons.
\newblock {\em Osaka J. Math.}, 48(2):497--513, 2011.

\bibitem{MR1768112}
G.~Tian and X.~Zhu.
\newblock Uniqueness of {K}\"{a}hler-{R}icci solitons.
\newblock {\em Acta Math.}, 184(2):271--305, 2000.

\bibitem{MR4437352}
A.~D. Vedova and A.~Gatti.
\newblock Almost {K}\"{a}hler geometry of adjoint orbits of semisimple {L}ie
  groups.
\newblock {\em Math. Z.}, 301(3):3141--3183, 2022.

\bibitem{MR2988734}
L.~Vezzoni.
\newblock A note on canonical {R}icci forms on {$2$}-step nilmanifolds.
\newblock {\em Proc. Amer. Math. Soc.}, 141(1):325--333, 2013.

\bibitem{MR2084775}
X.-J. Wang and X.~Zhu.
\newblock K\"{a}hler-{R}icci solitons on toric manifolds with positive first
  {C}hern class.
\newblock {\em Adv. Math.}, 188(1):87--103, 2004.

\bibitem{MR1817785}
X.~Zhu.
\newblock K\"{a}hler-{R}icci soliton typed equations on compact complex
  manifolds with {$C_1(M)>0$}.
\newblock {\em J. Geom. Anal.}, 10(4):759--774, 2000.

\end{thebibliography}

\end{document}